%% file: source.tex
\documentclass{amsart}
\usepackage[matrix,cmtip,arrow,curve]{xy}
\usepackage{tikz-cd}
\usepackage[unicode]{hyperref}
\usepackage{amssymb}
\usepackage{footnote}
\usepackage{amsmath}
\usepackage{amsthm}
\usepackage{amscd} 
\usepackage{enumerate}

\usepackage{tikz}
\usetikzlibrary{arrows,positioning}
\usepackage{tikz-cd}
\usetikzlibrary{decorations.pathmorphing}

\theoremstyle{plain}
\newtheorem{thm}{Theorem}[section]
\newtheorem{theorem}[thm]{Theorem}
\newtheorem{lem}[thm]{Lemma}
\newtheorem{prop}[thm]{Proposition}
\newtheorem{proposition}[thm]{Proposition}
\newtheorem{lemma}[thm]{Lemma}
\newtheorem{corollary}[thm]{Corollary}
\newtheorem*{mainthm}{Main Theorem}
\theoremstyle{definition}
\newtheorem{definition}[thm]{Definition}
\newtheorem{example}[thm]{Example}
\newtheorem{remark}[thm]{Remark}

\newtheorem{notation}[thm]{Notation}

\newcommand{\set}[1]{\{#1\}}
\newcommand{\Set}{\mathbf{Set}}
\newcommand{\Hom}{\operatorname{Hom}}

\newcommand{\obj}{\operatorname{Ob}}
\newcommand{\sCat}{\ensuremath{\mathbf{sCat}}}
\newcommand{\Operad}{\ensuremath{\mathbf{Operad}}}
\newcommand{\sOperad}{\ensuremath{\mathbf{sOperad}}}
\newcommand{\sMega}{\ensuremath{\mathbf{sMega}}}
\newcommand{\A}{\mathcal{A}}
\newcommand{\B}{\mathcal{B}}
\renewcommand{\C}{\mathcal{C}}
\newcommand{\D}{\mathcal{D}}
\newcommand{\T}{\mathcal{T}}
\newcommand{\R}{\mathcal{R}}
\newcommand{\someoperad}{\mathcal{O}}

\newcommand{\M}{\mathbb{M}}
\newcommand{\mst}{\mathcal{T}}
\DeclareMathOperator*{\colim}{\operatorname{colim}}
\newcommand{\Ho}{\operatorname{Ho}} 
\newcommand{\smallprod}[1]{(#1)}
\newcommand{\ua}{\underline a}
\newcommand{\ub}{\underline b}
\newcommand{\ba}{(\underline a; \underline b)}
\newcommand{\uamalg}[1]{\underset{#1}\amalg}
\newcommand{\fC}{\mathfrak{C}}
\newcommand{\uc}{\underline{c}}
\newcommand{\ud}{\underline{d}}
\newcommand{\sset}{\mathsf{sSet}}
\newcommand{\cat}{\mathsf{Cat}}
\newcommand{\dc}{\dch}
\newcommand{\xxsingle}{(x;x)}
\newcommand{\xysingle}{(y;x)}
\newcommand{\yxsingle}{(x;y)}
\newcommand{\yysingle}{(y;y)}
\newcommand{\profilev}{(\inp(v);\out(v))}
\newcommand{\dch}{(\uc;\ud)}
\DeclareMathOperator{\vertex}{Vt}
\DeclareMathOperator{\edge}{Edge}
\DeclareMathOperator{\image}{\operatorname{Im}}
\DeclareMathOperator{\inp}{in}
\DeclareMathOperator{\out}{out}
\DeclareMathOperator{\Ob}{Ob}
\DeclareMathOperator{\sk}{\mathsf{sk}}
\newcommand{\marked}{\mathsf{G}^{M}}
\newcommand{\incident}{\mathbb{Y}}
\newcommand{\id}{\operatorname{id}}
\newcommand{\sSet}{\mathbf{sSet}}
\newcommand{\Prop}{\ensuremath{\mathbf{Prop}}}
\newcommand{\sProp}{\ensuremath{\mathbf{sProp}}}
\newcommand{\Cat}{\ensuremath{\mathbf{Cat}}}
\newcommand{\Mega}{\ensuremath{\mathbf{Mega}}}
\newcommand{\X}{\mathcal{X}}
\newcommand{\Y}{\mathcal{Y}}
\newcommand{\col}{\operatorname{Col}}
\newcommand{\msh}{\mathcal{H}}
\newcommand{\msp}{\mathcal{P}}
\newcommand{\msq}{\mathcal{Q}}
\newcommand{\msi}{\mathcal{I}}
\newcommand{\notg}{{\not g}}
\newcommand{\blood}{\mathsf{B}}
\newcommand{\partition}{\boxplus}
\newcommand{\bloodadmissible}{\vphantom{\blood}^{a}\blood}
\newcommand{\bloodsimple}{\vphantom{\blood}^{s}\blood}
\newcommand{\properadG}{\mathcal{G}}
\newcommand{\properadW}{\mathcal{W}}
\newcommand{\properadB}{\mathcal{B}}
\newcommand{\sgc}{\mathsf{G}_{\fC}}
\newcommand{\op}{{\rm op}}
\newcommand{\transfusions}{\mathsf{T}}
\newcommand{\decAB}{D}
\newcommand{\calC}{\mathcal C}
\newcommand{\calD}{\mathcal D}
\newcommand{\Ononempty}{\mathbf{+}}
\newcommand{\precrush}{\widetilde{\msq}}
\newcommand{\sG}{\mathsf{G}}
\newcommand{\colorFunction}{\zeta}

\newcommand{\fD}{\mathfrak{D}}
\newcommand{\fdc}{(f\uc; f\ud)}
\newcommand{\profileg}{(\inp(G);\out(G))}

\title{The homotopy theory of simplicial props}

\author[P. Hackney]{Philip Hackney}
\address{Matematiska institutionen\\ Stockholms universitet\\ 106 91 Stockholm \\ Sweden}
\email{hackney@math.su.se} 

\author[M. Robertson]{Marcy Robertson}
\address{Department of Mathematics \\ The University of California Los Angeles \\ Los Angeles, CA}\email{mrober97@math.ucla.edu}


\begin{document}

\begin{abstract}The category of (colored) props is an enhancement of the category of colored operads, and thus of the category of small
categories. In this paper, the second in a series on `higher props,' we show that the category of all small colored simplicial props admits a cofibrantly generated model category structure. With this model structure, the forgetful functor from props to operads is a right Quillen functor.\end{abstract}

\maketitle


\input{graphs}

\input{intro}

\input{props}

\section{A review of various model category structures}\label{section:model structures}One of the most basic examples of a model category structure is the so-called `natural' model category structure on $\Cat$, the category of all small (non-enriched) categories.

\begin{theorem}\label{T:naturalcat} The category $\Cat$ admits a cofibrantly generated model category structure where:
\begin{itemize}
  \item the weak equivalences are the categorical equivalences;
  \item the cofibrations are the functors $F:\C\rightarrow\D$ which are injective on objects;
  \item the fibrations are the \emph{isofibrations}, i.e.\ functors $F:\C\rightarrow\D$ with the property that for each object $c$ in $\C$ and each isomorphism $f:Fc\rightarrow d$ in $\D$ there exists a $c'$ in $\C$ and an isomorphism $g:c\rightarrow c'$ in $\C$ such that $F(g)=f$.
\end{itemize}\end{theorem}

A proof of this theorem, along with a proof that $\Cat$ is a simplicial model category, may be found in \cite{rezkmodelcat}.
The Bergner model structure on simplicial categories blends the model category structure on $\Cat$ together with the standard model structure on simplicial sets~\cite{gj,QuillenHA}. 

Let $\sSet$ be the category of simplicial sets with the standard model structure~\cite{gj,QuillenHA}, and let $\sCat$ denote the category of \emph{simplicial categories}, by which we mean small categories enriched in $\sSet$. Given a simplicial category $\A$, we can form a genuine category $\pi_0(\A)$ which has the same set of objects as $\A$ and whose set of morphisms $\pi_{0}(\A)(x,y):=[*, \A(x,y)]$ is the set of path components of the simplicial set $\A(x,y).$ This induces a functor $\pi_{0}:\sCat\longrightarrow \Cat,$ with values in the category of small categories.

Given the conditions from Theorem~\ref{T:bergner}, it is immediate that the inclusion $\Cat \to \sCat$ is a right Quillen functor. Hence $\pi_0$ is a left Quillen functor,  so once we have established the model structure on $\sCat$ we will have the (total left derived) functor $\Ho(\sCat)\longrightarrow \Ho(\Cat)$ \cite[8.5.8(1)]{hirschhorn}. In other words, any $F:\C\longrightarrow\D$ in $\Ho(\sCat)$ induces a morphism $\pi_{0}(\C)\longrightarrow\pi_{0}(\D)$.
Note that this set up implies that the \emph{essential image} of a simplicial functor $F:\A\rightarrow\mathcal{B}$ must be defined as the full simplicial subcategory of $\mathcal{B}$ consisting of all objects whose image in the component category $\pi_{0}(\mathcal{B})$ are in the essential image of the functor $\pi_{0}(F)$.

Informally, one can say that inside every simplicial prop lies a simplicial category consisting of the 
operations with only one input and one output. More formally, there exists a forgetful functor 
$U_0:\sProp\rightarrow\sCat$ defined by $\obj U_0(\T) := \col(\T)$, $U_0(\T)(c,d):= \T(c;d)$. 
Composition is given by the applicable $\circ_v$ operations. The functor $U_0$ admits a left adjoint, 
$F_0$, which assigns a category $\C$ to an prop $F_0\C$ with $\col(F_0\C):=\obj(\C)$. 
This adjunction is factored by the adjunction $\sOperad \rightleftarrows \sProp$ from section~\ref{operadtoprop}.

We define the \emph{underlying component category} functor to be \begin{equation}\label{E:cptcat}
\sProp \overset{U_0}\to \sCat \overset{\pi_0}\to \Cat. \end{equation}When there is no ambiguity we will denote $
\pi_0(U_0(\T))$ by $\pi_0(\T)$. 

\subsection{The Bergner model structure on \texorpdfstring{$\sCat$}{sCat}}\label{S:bergnerMS}The following theorem (with `proper' replaced by `right proper') is due to Bergner~\cite{bergner}; left properness was shown by Lurie \cite[A.3.2]{htt}.

\begin{theorem}\label{T:bergner}The category of all small simplicial categories, $\sCat,$ supports a proper, cofibrantly generated, model category structure. The weak equivalences (respectively, fibrations) are the $\sSet$-enriched functors$$F:\mathcal{A}\longrightarrow\mathcal{B}$$ such that:
\begin{description}
  \item[W1 (F1)] for all objects $x,y$ in $\mathcal{A}$, the $\sSet$-morphism $F_{x,y}:\mathcal{A}(x,y)\longrightarrow\mathcal{B}(Fx,Fy)$ is a weak equivalence (respectively, fibration) in the model structure on $\sSet$, and
  \item[W2 (F2)]the induced functor $\pi_0(F):\pi_{0}(\mathcal{A})\longrightarrow \pi_{0}(\mathcal{B})$ is a weak equivalence (respectively, fibration) in $\Cat$.\end{description}\end{theorem}

Note that if $F:\A\rightarrow\B$ satisfies condition (W1), then $\pi_0(F)$ is automatically fully-faithful. Thus, checking that $F$ satisfies condition (W2) is equivalent to checking that the induced functor$$\pi_0(F):\pi_0(\A)\rightarrow\pi_0(\B)$$is essentially surjective.

Bergner gives an explicit description of the generating (acyclic) cofibrations of this model structure, which we will generalize in section~\ref{S:modelprop}. In order to describe these generating (acyclic) cofibrations we will need some background. Let $\varnothing$ denote the empty category and $\msi=\{x\}$ denote the terminal category, which is the category with one object and one identity arrow (viewed as a simplicial category by applying the strong monoidal functor $\Set\rightarrow\sSet$). Given any simplicial set $K$ we can generate a simplicial category $\mathcal{G}_{1,1}[K]$. The category $\mathcal{G}_{1,1}[K]$ has two objects, called $a_1$ and $b_1$, such that $\Hom(a_1,b_1)=K$, $\Hom(a_1,a_1) \cong \Hom(b_1, b_1) \cong *$, and $\Hom(b_1,a_1) = \varnothing$. 

For concreteness, we fix a set of 
generating cofibrations for $\sSet$ to be the boundary inclusions 
\begin{equation}\label{E:bdry}
\partial \Delta[n] \to \Delta[n]
 \qquad n\geq 0,
\end{equation} and a set of generating acyclic cofibrations to be the horn inclusions 
\begin{equation}\label{E:horni}
\Lambda[n,k] \to \Delta[n] \qquad n\geq 1, 0\leq k \leq n
\end{equation}
(see, for example, \cite[11.1.6]{hirschhorn} or \cite{QuillenHA}). The following two propositions appear in \cite{bergner}.

\begin{prop}\label{generatingcofibrationsbergner}A functor of simplicial categories $F:\A\rightarrow\B$ is an acyclic fibration if, and only if, $F$ has the right lifting property (RLP) with respect to all the maps\begin{itemize}
\item $\mathcal{G}_{1,1}[K]\hookrightarrow \mathcal{G}_{1,1}[L]$ where $K\hookrightarrow L$ is a generating cofibration of $\sSet$, and
\item the map $\varnothing\hookrightarrow\msi$.\end{itemize}\end{prop}

To describe the generating acyclic cofibrations of simplicial categories, Bergner looks at a family of simplicial categories $\msh$ which have two objects $x$ and $y$, weakly contractible function complexes, and only countably many simplices in each function complex. Furthermore, she requires $\msh$ to be cofibrant in the Dwyer-Kan model category structure on $\sCat_{\{x,y\}}$ \cite{dk-simploc}. Let $\mathbf{H}$ denote a set of representatives of isomorphism classes of such categories.

\begin{prop}\label{generatingtrivialcofibrationsbergner}A functor of simplicial categories $F:\A\rightarrow\B$ is a fibration if, and only if, $F$ has the right lifting property (RLP) with respect to maps of the form\begin{itemize}
\item $\mathcal{G}_{1,1}[K]\hookrightarrow \mathcal{G}_{1,1}[L]$ where $K\hookrightarrow L$ is a generating acyclic cofibration of $\sSet$, and
\item the maps $\msi\hookrightarrow\msh$ for $\msh \in \mathbf{H}$ which send $x$ to $x$.
\end{itemize}\end{prop}

\subsection{Categories of props with a fixed set of colors} Fix a set $\mathfrak{C}$ and consider the category of all simplicial props with a \emph{fixed} set of colors $\mathfrak{C}$. We will denote this category by $\sProp_{\mathfrak{C}}.$ In this category the morphisms are the morphisms of simplicial props $f:\R\rightarrow\T$ such that the induced map on colors $f:\col(\R)\rightarrow\col(\T)$ is the identity map on $\mathfrak{C}$. In particular, all of the morphisms in $\sProp_{\mathfrak{C}}$ are essentially surjective. This category admits a model structure analogous to the Dwyer-Kan model structure on simplicial categories with a fixed object set \cite{dk-simploc}. The following theorem was established (for more general enrichments) by Fresse in the monochrome case \cite{fresse} and Johnson and Yau \cite[3.11]{jy} in the fixed color case; alternatively it is a special case of \cite[2.1]{bmresolution} since $\sProp_{\mathfrak{C}}$ is a category of algebras over a $\M \mathfrak{C} \times \M \mathfrak{C}$-colored operad.

\begin{theorem}\label{T:jy} The category of simplicial props with fixed color set $\mathfrak{C}$ admits a cofibrantly generated model category structure where $f:\R\rightarrow\T$ is a weak equivalence (respectively, fibration) if for each input-output profile $\ba$ in $\mathfrak{C}$ the map$$f:\R(\ua;\ub)\longrightarrow\T(\ua;\ub)$$is a weak equivalence (respectively, fibration) of simplicial sets.\end{theorem}

Given a map of sets $\alpha:\mathfrak{C} \to \mathfrak{D}$ there is an induced adjoint pair of functors
\[\alpha_{!}:\sProp_{\mathfrak{C}}\rightleftarrows\sProp_{\mathfrak{D}}:\alpha^{*}.\]
If $\T$ is a prop with color set $\mathfrak{D}$, then the prop $\alpha^{*}\T$ has colors $\mathfrak{C}$ and operations
\[
        \alpha^{*}\T(\ua; \ub)=\T(\alpha \ua ; \alpha \ub).
\]
Notice that there is a canonical map $\alpha^*\T \to \T$ with color function $\alpha$, and any morphism $\R \to \T$ with color function $\alpha$ factors uniquely as $\R \to \alpha^*\T \to \T$.

\begin{prop}\cite[7.5]{jy}\label{prop37} The adjunction $(\alpha_!,\alpha^*)$ is a Quillen pair.
\end{prop}

\section{The model structure on \texorpdfstring{$\sProp$}{sProp}}\label{S:modelprop}

We now turn to the main result of the paper, namely the model structure on the category of simplicial props.

\begin{definition}\label{WEandfibrations} Let $\R$ and $\T$ be simplicial props, and let $f:\R \to \T$ be a morphism of simplicial props.
We say that $f$ is a \emph{weak equivalence} if
\begin{description}
\item[W1]\label{W1} for each input-output profile $\ba$ in $\col(\mathcal{R})$ the morphism$$f:\R(\ua; \ub)\longrightarrow\T(f\ua; f\ub)$$is a weak homotopy equivalence of simplicial sets; and
\item[W2]\label{W2} the functor $\pi_{0}f:\pi_{0}\R\rightarrow\pi_{0}\T$ from \eqref{E:cptcat} is an equivalence of categories. \end{description}
We say that the morphism $f$ is a \emph{fibration} if
\begin{description}\label{fibration}
\item [F1]\label{F1} for each input-output profile $\ba$ in $\col(\mathcal{R})$ the morphism$$f:\R(\ua; \ub))\longrightarrow\T(f\ua; f\ub)$$is a Kan fibration of simplicial sets; and
\item [F2]\label{F2} the functor $\pi_{0}f:\pi_{0}\R\rightarrow\pi_{0}\T$ is an isofibration.
\end{description}
\end{definition}

\begin{mainthm}\label{mainthm}The category of simplicial props $\sProp$ admits a cofibrantly generated model 
category structure with the above classes of weak equivalences and fibrations. Cofibrations are those 
morphisms which have the left lifting property (LLP) with respect to the acyclic fibrations.
\end{mainthm}

In the special case where we consider simplicial categories as simplicial props, the conditions (W1) and (W2) are precisely those conditions that give a weak equivalence of simplicial categories in the Bergner model structure.
In the case where $\R$ and $\T$ are two simplicial props with the same set $\mathfrak{C}$ of 
colors and $f$ acts as the \emph{identity} of $\mathfrak{C}$, then condition (W1) implies (W2). Thus 
such an $f$ is a weak equivalence if and only if it is a weak equivalence in the Johnson-Yau model 
structure on $\sProp_{\mathfrak{C}}$ from Theorem~\ref{T:jy}.

We further note that there is a similar model structure on the category of simplicial colored operads, $\sOperad$, which was shown by the second author \cite{robertson1} and independently by Cisinski and Moerdijk \cite{cm-simpop}.
Notice that the model structure on $\sProp$ is \emph{not} lifted from the model structure on $\sOperad$ across the adjunction $(F_0,U_0)$.
\begin{corollary}\label{cor42} The adjunction $F_0: \sOperad \rightleftarrows\sProp: U_0$ is a Quillen adjunction. \end{corollary}
\begin{proof} The forgetful functor $U_0$ preserves fibrations and acyclic fibrations.
\end{proof}

We provide an explicit description of the generating cofibrations and generating acyclic cofibrations.  For this purpose, we will need the following construction. Each $(n,m)$-corolla determines a prop $\mathcal{G}_{n,m}$, as described in Definition~\ref{D:gnm}. 
Then, given a simplicial set $X$ and integers $n,m\geq 0$, we denote by $\mathcal{G}_{n,m}[X]$ the free simplicial prop with $(n,m)$-ary operations decorated by the simplicial set $X$. 
In other words, the simplicial prop $\mathcal{G}_{n,m}[X]$ has colors $a_1,...,a_n;b_1,\ldots,b_m$ and
\begin{equation}\label{eq:gnm}
\mathcal{G}_{n,m}[X](a_1,\ldots,a_n;b_1,...,b_m)=X. 
\end{equation}
This construction is functorial in $X$.

\begin{remark}\label{R:characterization} The props $\mathcal{G}_{n,m}[X]$ are characterized by the property that a map $f:\mathcal{G}_{n,m}[X]\rightarrow \T$ consists of a set map \[ \set{a_1,\dots,a_n, b_1\dots,b_m}\rightarrow\col\T\]  together with a simplicial set map \[ f:X\rightarrow \T(fa_1,\dots,fa_n, fb_1\dots,fb_m).\]\end{remark}

Recall that $\msi$ is the category with one object $x$ and no non-identity morphisms and that we have fixed a specific set of generating (acyclic) cofibrations for $\sSet$ \eqref{E:bdry}, \eqref{E:horni}.

\begin{definition}[Generating cofibrations]\label{generatingcofibrations}The set $I$ of generating cofibrations consists of the following morphisms of simplicial props:
\begin{description}

        \item[C1] Given a generating cofibration $K\hookrightarrow L$ in the model structure on $\sSet$, the induced morphisms of props $\mathcal{G}_{n,m}[K]\longrightarrow \mathcal{G}_{n,m}[L]$ for each $n,m \in\mathbb{N}$.

        \item[C2] The 
        map $\varnothing \hookrightarrow F_0(\msi)$; we will usually write this as just $\varnothing \hookrightarrow \msi$ when no confusion will result.
\end{description}
\end{definition}

As in section~\ref{S:bergnerMS}, we consider simplicial categories $\msh$ with two objects $x$ and $y$, weakly contractible function complexes, and only countably many simplices in each function complex. Furthermore, we require that each such $\msh$ is cofibrant in the Dwyer-Kan model category structure on $\sCat_{\{x,y\}}$ \cite{dk-simploc}. Let $\mathbf{H}$ denote a set of representatives of isomorphism classes of such categories.

\begin{definition}[Generating acyclic cofibrations]\label{generatingacycliccofibrations}The set $J$ of generating acyclic cofibrations consists of the following morphisms of simplicial props:
\begin{description}
        \item[A1] Given a generating acyclic cofibration $K\hookrightarrow L$ of the model structure on $\sSet$, the morphisms $\mathcal{G}_{n,m}[K]\rightarrow \mathcal{G}_{n,m}[L]$ for each $n,m\in\mathbb{N}$.
        \item[A2] 
        The maps $F_0(\msi\hookrightarrow\msh)$ for $\msh \in \mathbf{H}$ which take $x$ to $x$; as in (C2), we will abuse notation and denote this map by $\msi\hookrightarrow\msh$.
\end{description}
        \end{definition}

Note that all of the morphisms in (A1) are weak equivalences in the Johnson-Yau model structure on $\sProp_{\set{a_1,\dots,a_n;b_1,\dots,b_m}}$ and thus are weak equivalences in $\sProp$.

The proof of our main theorem applies Kan's recognition theorem for 
cofibrantly generated model categories (see \cite[11.2.1]{hirschhorn} or \cite[2.1.19]{hovey}). 
We begin by verifying compatibility of Definitions~\ref{generatingcofibrations} and \ref{generatingacycliccofibrations} with our (acyclic) fibrations from \ref{WEandfibrations}.

\subsection{Classification of (acyclic) fibrations}\label{subsection classification fibrations}

As we will show later, the saturation of the set of generating acyclic cofibrations $J$, denoted by $\overline{J},$ is the class of acyclic cofibrations in $\sProp$. Similarly, the saturation of the set of generating cofibrations $I$, 
denoted by $\overline{I},$ is the class of cofibrations in $\sProp$. In the following two lemmas we show that this is consistent with our choice of fibrations and acyclic fibrations. 

\begin{lemma}\label{Classification of Fibrations} A morphism of simplicial props $f:\R\rightarrow \T$ has a the right lifting property with respect to the class $\overline{J}$ if, and only if, $f$ is a fibration of simplicial props. \end{lemma}

\begin{proof}\label{Proof Fibrations}
Let $K\hookrightarrow L$ be a generating acyclic cofibration of $\sSet$. 
Remark~\ref{R:characterization} implies that a morphism of simplicial props $f:\R\rightarrow \T$ 
has the right lifting property (RLP) with respect to $j:\mathcal{G}_{n,m}[K]\hookrightarrow \mathcal{G}_{n,m}[L]$ for some 
$n,m\in\mathbb{N}$ if, and only if, for each input-output profile $c_1,\dots,c_n; d_1,\dots, d_m$ in $\R$, 
the map of simplicial sets $f:\R(\uc;\ud)\longrightarrow \T(f\uc; f\ud)$ has the RLP with respect to $K\hookrightarrow L$. Thus 
$f$ satisfies (F1) if and only if $f$ has the RLP with respect to every map $j$ in (A1).

Assume now that $f$ has the RLP with respect to a prop morphism in the set (A2). To avoid confusion with the simplicial functors $\msi\hookrightarrow\msh$, we will write $F_0 \msi \hookrightarrow F_0 \msh$ when considering maps of simplicial props in (A2) for the remainder of the proof.

If $f$ is a fibration, then $U_0(f)$ is a fibration in $\sCat$, so $U_0(f)$ has the RLP with respect to all of the simplicial functors $\msi \hookrightarrow \msh$ from Proposition~\ref{generatingtrivialcofibrationsbergner}. A commutative square
\[
        \xymatrix{
                F_0\msi \ar@{->}[r] \ar@{->}[d] &
                \R \ar@{->}[d] \\
                F_0\msh \ar@{->}[r] &
                \T
        }
\]
factors as
\[
        \xymatrix{
                F_0\msi \ar@{->}[r] \ar@{->}[d] & F_0 U_0 \R \ar@{->}[r] \ar@{->}[d] &
                \R \ar@{->}[d] \\
                F_0\msh \ar@{.>}[ur] \ar@{->}[r] & F_0 U_0 \T \ar@{->}[r] &
                \T
        }
\]
where the dotted lift exists since $U_0F_0 = \id_{\sCat}$ and $U_0(f)$ has the RLP with respect to $\msi \hookrightarrow \msh$. Thus $f$ has the RLP with respect to maps in (A2).
On the other hand, if $f$ has the RLP with respect to the maps $F_0\msi \to F_0\msh$ in (A2) and we have a commutative square
\[
        \xymatrix{
                \msi \ar@{->}[r] \ar@{->}[d] &
                U_0\R \ar@{->}[d] \\
                \msh \ar@{->}[r] &
                U_0\T,
        }
\]
we know that we have a dotted lift in
\[
        \xymatrix{
                F_0\msi \ar@{->}[r] \ar@{->}[d] & F_0 U_0 \R \ar@{->}[r] \ar@{->}[d] &
                \R \ar@{->}[d] \\
                F_0\msh \ar@{.>}[urr] \ar@{->}[r] & F_0 U_0 \T \ar@{->}[r] &
                \T.
        }
\]
This map factors through $F_0U_0\R$, and applying $U_0$ we have the desired lift $\msh = U_0F_0 \msh \to U_0F_0U_0 \R = U_0\R$. Thus $U_0(f)$ has the RLP with respect to maps of the form $\msi \to \msh.$ If, furthermore, we assume that $f$ has the RLP with respect to maps in (A1), then, in particular, it has the RLP with respect to the maps $\mathcal{G}_{1,1}[K] \to \mathcal{G}_{1,1}[L]$, and so $U_0(f)$ is a fibration in $\sCat$ by Proposition~\ref{generatingtrivialcofibrationsbergner}. Thus $\pi_0(U_0(f)) = \pi_0(f)$ is a fibration of categories, and we already knew that $f$ satisfied (F1), so $f$ is a fibration in $\sProp$.
\end{proof}

\begin{lemma}\label{Classification of Acyclic Fibrations} A morphism $f:\R\rightarrow \T$ has the right lifting property with respect to the class $\overline{I}$ if, and only if, $f$ is an acyclic fibration of simplicial props. \end{lemma} 

\begin{proof} The proof is nearly identical to that of Lemma~\ref{Classification of Fibrations}.\end{proof}

\subsection{Relative \texorpdfstring{$J$}{J}-cell complexes and proof of the main theorem}\label{subsection relative cell complexes}
The subcategory of \emph{relative $J$-cell complexes} consists of those maps which can be constructed as a transfinite composition of pushouts of elements of $J$.

\begin{prop}\label{The Hard Part}Every relative $J$-cell complex is a weak equivalence.\end{prop}

It is enough to show that a pushout of a map in $J$ is a weak equivalence.  
In this section we will analyze the maps (A1) from Definition \ref{generatingacycliccofibrations}. 
Analysis of the maps (A2) is considerably more involved, and will be deferred until sections \ref{section specialization}--\ref{section filtration layers}.

\begin{prop} Let $f:\B \to \R$ be a morphism of props with color map $\alpha$. Then there is a unique map $\tilde f : \alpha_! \B \to \R$ so that $\B \to \R$ factors as $\B \to \alpha_! \B \to \R$. This decomposition is functorial.
\end{prop}
\begin{proof}
Recall that $\alpha^*\R (\ua; \ub) = \R(\alpha\ua; \alpha\ub)$ and that the map $\B \to \alpha_! \B$ is defined by
\[ \B \to \alpha^* \alpha_! \B \to \alpha_! \B\]
where the first map is the unit of the adjunction.
Write
\[
	\Phi: \Hom(\alpha_! \B, \R) \overset\cong\to \Hom(\B, \alpha^* \R)
\]
for the adjunction isomorphism.
Existence and uniqueness follow from the diagram
\[
        \xymatrix{
                \B \ar@{->}[r]^-{unit} \ar@{->}[dr]_{\Phi(\tilde f)} &\alpha^*\alpha_!\B \ar@{->}[r] \ar@{->}[d]^{\alpha^* \tilde f} &
                \alpha_!\B \ar@{->}[d]^{\tilde f} \\
                & \alpha^*\R \ar@{->}[r] &
                \R
        }
\]
and the fact that  
$\Phi(\tilde{f})$ must be defined by
\[ \begin{tikzcd}
	\,	& \R(\alpha \ua;\alpha \ub) \\
       \B(\ua; \ub) \arrow{ur}{f} \rar[swap]{\Phi(\tilde f)} &  \alpha^*\R(\ua; \ub) \uar[leftrightarrow]{\cong}
\end{tikzcd}\]
since the composite $\B \overset{\Phi(\tilde f)}\longrightarrow \alpha^*\R \to \R$ needs to be $f$. 

\end{proof}

\begin{prop}\label{P:pushout}
Let $j: \A \to \B$ be a morphism in $\sProp_{\mathfrak{C}}$ and $\alpha:\mathfrak{C} \to \mathfrak{D}$ a map of sets. Then
\[
        \xymatrix{
                \A \ar@{->}[r] \ar@{->}[d]^j &
                \alpha_!\A \ar@{->}[d]^{\alpha_!j} \\
                \B \ar@{->}[r] &
                \alpha_!\B
        }
\]
is a pushout in $\sProp$.
\end{prop}
\begin{proof}
Suppose we have a diagram
\[
        \xymatrix{
                \A \ar@{->}[r] \ar@{->}[d] &
                \alpha_!\A \ar@{->}[d] \ar@/^/[ddr] \\
                \B \ar@{->}[r] \ar@/_/[drr]_{x}&
                \alpha_! \B \ar@{.>}[dr] \\
                && \R
        }
\]
where $\alpha_!\A \to \R$ has color map $\beta$.
If $\beta=\id_{\mathfrak{D}}$, then the dotted arrow exists and is unique by the previous proposition.

If $\beta$ is not the identity, then $x : \B \to \R$ has color map $\beta\alpha$. Consider the diagram
\[
        \xymatrix{
                \A \ar@{->}[r] \ar@{->}[d]^j &
                \alpha_!\A \ar@{->}[d]^{\alpha_!j} \ar@{->}[r]  &
                \beta_!\alpha_! \A \ar@/^/[ddr] \ar@{->}[d] \\
                \B \ar@{->}[r]^i \ar@/_/[drrr]_{x} &
                \alpha_!\B  \ar@{->}[r]^{i'} \ar@{.>}[drr]|q &
                (\beta\alpha)_! \B  \ar@{->}[dr]^{\tilde x} &  \\
                &&& \R
        }
\]
The arrow $\tilde x : (\beta\alpha)_!\B \to \R$ exists by the previous paragraph, which implies that the dotted arrow $q$ exists as well. If we had some other map, say $q_0$, which made the diagram
\[
        \xymatrix{
                \A \ar@{->}[r] \ar@{->}[d] &
                \alpha_!\A \ar@{->}[d] \ar@/^/[ddr] \\
                \B \ar@{->}[r]^i \ar@/_/[drr]_{x} &
                \alpha_! \B \ar@{.>}[dr]|{q_0} \\
                && \R
        }
\]
commute, then we can factor $q_0$ uniquely as
\[ q_0 : \alpha_! \B \overset{i'}\to \beta_! \alpha_! \B \overset{\tilde q_0}\to \R. \] 
Thus $x = q_0i = \tilde q_0 i'i$.
Of course $x$ factors uniquely as
\[ x : \B \overset{i'i}\to (\beta\alpha)_! \B \overset{\tilde x} \to \R,\]
so we see that $\tilde q_0 = \tilde x$, implying that $q_0=\tilde q_0 i' = \tilde x i' =  q$.

\end{proof}

\begin{lem}\label{L:newcolorchange}
Consider the following pushout square in the category of simplicial props
$$\begin{CD}
\A@>f>> \A'\\
@VVjV        @VVgV\\
\B@>k>> \B'. 
\end{CD}$$ If $j:\A\longrightarrow\B$ is an acyclic cofibration of $\sProp_{\fC}$ and $f:\A\rightarrow\A'$ is a morphism of simplicial props, then $g:\A'\longrightarrow\B'$ is a weak equivalence in $\sProp$.
\end{lem}
\begin{proof} 

Let $\alpha := \col(f): \mathfrak{C} \to \mathfrak{D} = \col(\A')$.
Since $j:\A \to \B$ is an acyclic cofibration in $\sProp_{\mathfrak{C}},$ $\alpha_!j$ is an acyclic cofibration in $\sProp_{\mathfrak{D}}$ by Proposition~\ref{prop37}. We know that the left hand square in
\[
        \xymatrix{
                \A \ar@{->}[r] \ar@{->}[d]^j &
                \alpha_!\A \ar@{->}[d]^{\alpha_!j} \ar@{->}[r] &
                \A' \ar@{->}[d]^g \\
                \B \ar@{->}[r] &
                \alpha_!\B \ar@{->}[r] &
                \B'
        }
\]
is a pushout by Proposition~\ref{P:pushout} and the full rectangle is a pushout by assumption, so the right rectangle is a pushout in $\sProp$ as well.
This implies that the right hand rectangle is a pushout in $\sProp_{\mathfrak{D}}$, proving that $g$ is an acyclic cofibration in $\sProp_{\mathfrak{D}}$. Since $g$ is a weak equivalence in $\sProp_{\mathfrak{D}}$, $g$ is, in particular, a weak equivalence in $\sProp.$
\end{proof}

The following proposition is quite subtle, and forms the technical heart of this paper.

\begin{prop}\label{NAFA} Let $\T$ be a simplicial prop, $f:\msi\to \T$ a morphism of simplicial props, and let $j:\msi \to \msh$ be a generating acyclic cofibration. Given the pushout diagram
\[ 
        \xymatrix{
                \msi \ar@{->}[r]^{f} \ar@{->}[d]_{j} & \mst \ar@{->}[d]^{g} \\
                \msh \ar@{->}[r] & \msp
        }
\]
the map $g$ satisfies (W1).

\end{prop}

We will present a proof of Proposition~\ref{NAFA} beginning in section~\ref{S:model for pushouts} and continuing until the end of the paper.
Here is an alternative suggestion for a method of proof, which we've only been able to work out parts of, and which presents considerable technical difficulties of its own. 
One would like to be able to adapt the proof of \cite[3.4]{bergner} to the prop setting, but a major hurdle is that the category of props is not left proper (see \cite{leftProper}).
On the other hand, we suspect that this category is `relatively left proper' in the sense of \cite{leftProper}, i.e.\ that pushouts of weak equivalences between $\Sigma$-cofibrant props are again weak equivalences.
This is likely enough to modify the proof in \cite{bergner} to the present setting. 

\begin{proof}[Proof of Proposition~\ref{The Hard Part}, assuming Proposition~\ref{NAFA}]
It is enough to show that the pushout of a generating acyclic cofibration is a weak equivalence.
In other words, given $j:\A\rightarrow\mathcal{B}$ which is in either the set (A1) or the set (A2), if
$$\begin{CD}
\A @>f>> \T\\
@VVjV        @VVgV\\
\mathcal{B} @>k>> \mathcal{P}
\end{CD}$$ is a pushout square in $\sProp$, then $g$ is a weak equivalence.
First, let us assume that $j$ is in (A1) and thus is of the form $j: \mathcal{G}_{n,m}[K] \to \mathcal{G}_{n,m}[L]$ where $K\to L$ is a generating acyclic cofibration of $\sSet$ (see \eqref{E:horni}).

Then $j: \mathcal{G}_{n,m}[K] \to \mathcal{G}_{n,m}[L]$ is an acyclic cofibration in $\sProp_{\{a_1, \dots, a_n, b_1, \dots, b_m\}}$ (since fibrations are defined levelwise), so $g$ is a weak equivalence by Lemma~\ref{L:newcolorchange}.

Now we assume that $j$ is an acyclic cofibration from set (A2) and consider the following pushout square: $$\begin{CD}
\msi  @>>f>       \T \\
@VVjV               @VVgV\\
\msh   @>>k>       \msp.
\end{CD}$$
We know that $g$ satisfies (W1) by Proposition~\ref{NAFA}, which implies that $\pi_0(g)$ is a fully faithful functor. Thus, to show that $g$ satisfies (W2), we must show that $\pi_0(g)$ is essentially surjective. But notice that
\[
        \obj \pi_0 \msp = \obj U_0 \msp = \col \msp = (\obj \pi_0 \T) \amalg \set{y},
\]
so we just need to show that $y$ is isomorphic to an object in the image of $\pi_0(g)$; we will show that $y\cong \pi_0(g) x$.
Let $a\in \msh(x,y)_0$ and $b\in \msh(y,x)_0$. Then since $\msh(x,x)$ and $\msh(y,y)$ are weakly contractible, $a\circ b$ is in the same path component as $\id_x$ and $b\circ a$ is in the same path component as $\id_y$.
Hence $k(a\circ b) = k(a) \circ k(b)$ is in the same path component as $k(\id_x) = \id_{\pi_0(g)x}$ and $k(b) \circ k(a)$ is in the same path component as $\id_y$. Hence $y \cong \pi_0(g) x$ in $\pi_0 \msp.$
\end{proof}

We can now prove the existence of the model category structure on the category of all small simplicial props.

\begin{proof}[Proof of Main Theorem, assuming Proposition~\ref{NAFA}]\label{proofofmaintheorem}
To prove the main theorem we will apply Kan's recognition theorem for cofibrantly generated model categories; we will indicate why the conditions of \cite[2.1.19]{hovey} are satisfied.

The category $\Prop$ is bicomplete (see \cite{hackneyrobertson1}), hence so is $\sProp$.
Weak equivalences are closed under retracts and satisfy the ``2-out-3'' property using the corresponding properties from $\sSet$ and $\Cat$.
The domains of $I$ and $J$, namely $\mathcal{G}_{n,m}[\partial \Delta[p]]$, $\mathcal{G}_{n,m}[\Lambda[p,k]]$, $\varnothing$, and $\msi$, are all small \cite[2.1.3]{hovey}.
Thus (1), (2), and (3) hold.
Our classification of fibrations and acyclic cofibrations, Lemmas~\ref{Classification of Fibrations} and \ref{Classification of Acyclic Fibrations}, imply (5) and (6).
Finally, (4) follows directly from Proposition~\ref{The Hard Part} and the classification of fibrations given in Lemma~\ref{Classification of Fibrations}.
\end{proof}

\subsection{\texorpdfstring{$\sProp$}{sProp} is right proper}\label{subsection right proper}

The model structure on $\sCat$ from Theorem~\ref{T:bergner} was right proper, and we now show that the same is true for the model structure on $\sProp$. 
The category $\sProp$ is not left proper, due to a counterexample of Dwyer given in \cite{leftProper}.

\begin{lemma}
Let
\[
        \xymatrix{
                \A \ar@{->}[r]^h \ar@{->}[d]^f &
                \R \ar@{->}[d]^g \\
                \B \ar@{->}[r]^k &
                \T
        }
\]
be a pullback square in $\sProp$. Then, for any input-output profile $\dc$ in $\A$, the diagram
\[
        \xymatrix{
                \A(\uc; \ud) \ar@{->}[r]^h \ar@{->}[d]^f &
                \R(h\uc; h\ud) \ar@{->}[d]^g \\
                \B(f\uc; f\ud) \ar@{->}[r]^k &
                \T(gh\uc; kf\ud)
        }
\]
is a pullback in $\sSet$.
\end{lemma}
\begin{proof} A commutative diagram of the form
\[
        \xymatrix{
                X \ar@/^/[rrd] \ar@/_/[ddr] \\ &
                \A(\uc;\ud) \ar@{->}[r]^h \ar@{->}[d]^f &
                \R(h\uc;h\ud) \ar@{->}[d]^g \\ &
                \B(f\uc;f\ud) \ar@{->}[r]^k &
                \T(gh\uc;kf\ud)
        }
\]
gives
\[
        \xymatrix{
                \mathcal{G}_{n,m}[X]  \ar@/^/[rrd] \ar@/_/[ddr] \ar@{.>}[dr]^p\\ &
                \A \ar@{->}[r]^h \ar@{->}[d]^f &
                \R \ar@{->}[d]^g \\ &
                \B \ar@{->}[r]^k &
                \T
        }
\]
and so $p$ induces the necessary map $X\to \A(\uc;\ud)$.
\end{proof}

\begin{prop}
The model category structure on $\sProp$ is right proper. In other words, every pullback of a weak equivalence along a fibration is a weak equivalence.
\end{prop}
\begin{proof}
Let
\[
        \xymatrix{
                \A \ar@{->}[r]^h \ar@{->}[d]^f &
                \R \ar@{->}[d]^g \\
                \B \ar@{->}[r]^k &
                \T
        }
\]
be a pullback square in $\sProp$ with $g$ a weak equivalence and $k$ a fibration.

        Since $U_0$ is a right adjoint it preserves pullbacks, and, moreover, is a right Quillen functor so preserves fibrations as well. By inspection $U_0$ preserves weak equivalences. Thus $U_0(f)$ a weak equivalence in the model structure on $\sCat$, so $f$ satisfies (W2).
        The previous lemma and the fact that  $\sSet$ is right proper \cite[13.1.13]{hirschhorn} imply that $f$ satisfies (W1).
\end{proof}

\section{A model for the pushout in simplicial props}\label{S:model for pushouts}

Recall, from Theorem \ref{sgc is prop}, the $\Cat$-enriched $\fC$-colored prop $\sgc = \{ \sgc\dc \}$.
The objects of the category $\sgc\dc$ are $\fC$-colored graphs and morphisms are defined by graph substitution
In this section we introduce a variant of of $\sgc\dc$ to obtain a prop enriched in $\Cat$ which models pushouts of props.
Suppose that we want to model a pushout 
\[
\begin{tikzcd} \properadG \dar{m_w} \rar{m_b} & \properadB \\
\properadW 
\end{tikzcd}
\]
where $\mathcal{W}$, $\mathcal{G}$ and $\mathcal{B}$ are simplicial props.  We know that the pushout simplicial prop will have color set $\fC = \col (\properadW) \amalg_{\col (\properadG)} \col (\properadB)$.  
The operations, however, are complicated; each operation is represented by a $\fC$-colored graph where each vertex is labelled by an operation from either $\mathcal W$, $\mathcal G$, or $\mathcal B$.
To help visualize these operations, imagine that said vertices are colored by either white, grey, or black. 
We now introduce a $\Cat$-enriched $\fC$-colored prop $\blood_{\fC}$ which has $\fC$-colored graphs with $\{g,w,b\}$-labelled vertices as objects and morphisms defined by graph substitution.

\begin{definition}
Let $\fC$ be a set and let $\dch$ be some fixed input-output profile. Define a category $\blood\dc$ as follows.
\begin{itemize}
\item The objects of $\blood\dc$ are graphs
$G\in \sgc\dc$ together with a partition function $\partition : \vertex(G) \to \{g,w,b \}$, coloring each vertex either grey, white, or black.

\item A \emph{gsd map} $(G,\partition) \to (J,\partition')$ is a graph substitution decomposition
$G=J\{ K_v \}_{v\in \vertex(J)}$ of $G$ so that $\partition(v') = \partition'(v)$ for every $v' \in \vertex(K_v)$.

\item Given such a gsd map, there is a map of sets
\[
	f_{\vertex}: \vertex(G) \to \vertex(J)
\]
which takes every vertex in $K_v$ to $v$.

\item A \emph{transfusion} $q: \partition \to \partition'$ of partitions of $G$ turns some grey vertices into either white vertices or black vertices. There is at most one transfusion between any two partitions.
The requirement to have a transfusion is that if $\partition(v) \in \{w,b\}$, then $\partition'(v) = \partition(v)$.

\item Morphisms of $\blood\dc$ consist of a transfusion followed by a gsd map.
We write such morphisms as $(f,q)$ where $q: \partition' \to \partition$ is a transfusion and $f$ is a gsd map on $(G,\partition)$; if $f$ or $q$ is an identity, we will write $(f,q)$ as just $q$ or $f$, respectively.
\end{itemize}
\end{definition}

Let us now define the composition in $\blood\dc$.
Given two morphisms of $\blood\dc$
	\begin{gather*}
		(G,\partition_1) \overset{q'}\to (G,\partition_2) \overset{f}\to (J,\partition_2') \\
		(J,\partition_2') \overset{q}\to (J,\partition_3) \overset{g}\to (K,\partition_3'),
	\end{gather*}
where $q',q$ are transfusions and $f,g$ are gsd maps,
	the composition in $\blood\dc$ 
	is defined by
	considering the composite
	\begin{equation}\label{fstardef}
		f^* \partition_3 \colon \vertex(G) \overset{f_{\vertex}}\to \vertex(J) \overset{\partition_3}\to \{ g, w, b\},
	\end{equation}
	filling in the diagram (which we do momentarily)
	\begin{equation}\label{composition def} \begin{tikzcd}
	(G,\partition_1) \rar{q'} & 
	(G,\partition_2) \dar{f} \rar[color=red]{f^*q} &
	\color{red} (G,f^*\partition_3) \dar[color=red]{f} \\ &
	(J,\partition_2') \rar{q} & (J,\partition_3) \dar{g} \\
	&& (K,\partition_3')
	\end{tikzcd} \end{equation}
	and setting $(g,q)\circ (f,q') = (g{\color{red}f}, {\color{red}(f^*q)}q')$.
	To see that there is a transfusion $f^*q \colon \partition_2 \to f^* \partition_3$, note that if $\partition_2(v) \in \{w,b\}$, then $\partition'_2 f_{\vertex}(v) = \partition_2(v) \in \{ w, b\}$, hence we have the middle equality in
	\[
		\left(f^*\partition_3\right)(v) = \partition_3 f_{\vertex} (v) = \partition'_2 f_{\vertex}(v) = \partition_2 (v).
	\]

\begin{proposition}\label{Biscategory}
	With the composition given above, $\blood\dc$ is a category.
\end{proposition}

\begin{proof} 

We will show that composition is associative, i.e.\ that $(f_3,q_3) \circ (f_2,q_2) \circ (f_1,q_1)$ doesn't depend on choice of parenthesization.
In the diagram
\[\begin{tikzcd}
	\color{orange}(G,\partition_0) \rar[color=orange]{q_1} & (G,\partition_1) \dar{f_1} \\
	& (J, \partition_1') \rar{q_2} & (J, \partition_2) \dar{f_2} \\
	&& (K,\partition_2') \rar{q_3} & (K,\partition_3) \dar[color=orange]{f_3} \\
	&&&\color{orange}(L,\partition_3')
\end{tikzcd}
\]
only the black portion is relevant for this task.
Unraveling, we have 
\[\begin{tikzcd}
	(G,\partition_1) \dar{f_1} \rar[color=red]{f_1^*q_2} & \color{red} (G,f_1^*\partition_2) \dar[color=red]{f_1} \rar[color=purple,dashed]{(f_2f_1)^*q_3} & \color{purple} (G,(f_2f_1)^*\partition_3) \arrow[color=purple, dashed]{dd}{f_2f_1}\\
	(J, \partition_1') \rar{q_2} & (J, \partition_2) \dar{f_2} \\
	& (K,\partition_2') \rar{q_3} & (K,\partition_3) 
\end{tikzcd}\]
and
\[\begin{tikzcd}
	(G,\partition_1) \dar{f_1} \arrow[color=purple,dashed]{rr}{f_1^*(f_2^*q_3 \circ q_2)} && \color{purple} (G,f_1^*(f_2^*\partition_3)) \dar[color=purple,dashed]{f_1}\\
	(J, \partition_1') \rar{q_2} & (J, \partition_2) \dar{f_2} \rar[color=red]{f_2^*q_3} & \color{red} (J,f_2^*\partition_3) \dar[color=red]{f_2} \\
	& (K,\partition_2') \rar{q_3} & (K,\partition_3) 
\end{tikzcd}\]
defining the two ways to compose.
But there is at most one transfusion between two partitions, hence we must just show that
\[
	(f_2f_1)^*\partition_3 = f_1^*(f_2^*\partition_3).
\]
This follows immediately from the definition \eqref{fstardef}. 

\end{proof}

\begin{theorem}\label{thm_bloodprop} If we allow the input-output profiles $\dch$ to vary, graph substitution produces a $\fC$-colored prop $\blood=\blood_{\fC}$ enriched in $\cat$.
\end{theorem}

\begin{proof} This is a variation on Proposition \ref{sgc is prop}.
If $J$ is a $\fC$-colored graph and we have a collection of objects $(K_v, \partition_v) \in \blood\profilev$ ranging over the vertices of $J$, then the graph substitution
	\[
		G = J\{ K_v \}
	\]
	is an object of $\blood(\inp(J);\out(J))$ 
	by considering the partition
	\[ 
	\left(\coprod_{\vertex(J)} \partition_v \right) : \vertex(G) = \coprod_{\vertex(J)} \vertex(K_v) \to \{g,w,b\}. \] 
\end{proof}

Suppose we are given a pushout diagram in $\sProp$ that we wish to compute. Then we define decoration functors $D_{\dch}:\blood\dc\rightarrow\sset$ 
so that $\{ \colim D_{\dch} \}$ will be the underlying megagraph of the pushout.

\begin{definition} 
Suppose we have two prop morphisms
\[ \properadW \overset{m_w}\leftarrow \properadG \overset{m_b}\rightarrow \properadB, \]
and set $\fC = \col (\properadW) \amalg_{\col (\properadG)} \col (\properadB)$.
For a given input-output profile $\dch$ we describe a \emph{decoration functor}
$D_{\dch}: \blood\dc \to \sset$ as follows. 

\begin{enumerate} 

\item On objects of $\blood\dc$ we define $D_{\dch}(G,\partition)$ to be 
\[
	\left[\prod_{\partition(v) = w} \properadW\profilev \right] \times \left[\prod_{\partition(v) = b} \properadB\profilev \right]  \times \left[\prod_{\partition(v) = g} \properadG\profilev \right]. 
\]

\item On transfusions $\partition \to \partition'$, we define $D_{\dch}(G,\partition) \to D_{\dch}(G,\partition')$ by applying $m_w$ and $m_b$ to the $\properadG$-decorations of the grey vertices of $G$
\begin{align*}&\left[\prod_{\partition(v) = g} \properadG\profilev \right] \\
\cong& \left[
	\prod_{\partition'(v)= g} \properadG\profilev \right] \times \left[
	\prod_{\substack{\partition(v) = g \\ \partition'(v)=w}} \properadG\profilev \right] \times \left[
	\prod_{\substack{\partition(v) = g \\ \partition'(v)=b}} \properadG\profilev \right]  \\
\overset{\id \times m_w \times m_b}\to& \left[
	\prod_{\partition'(v)= g} \properadG\profilev \right] \times \left[
	\prod_{\substack{\partition(v) = g \\ \partition'(v)=w}} \properadW\profilev \right] \times \left[
	\prod_{\substack{\partition(v) = g \\ \partition'(v)=b}} \properadB\profilev \right] 
\end{align*}
and then rearranging to get a map
\begin{multline*}\left[\prod_{\partition(v) = w} \properadW\profilev \right] \times \left[\prod_{\partition(v) = b} \properadB\profilev \right]  \times \left[\prod_{\partition(v) = g} \properadG\profilev \right] \\
\to \left[\prod_{\partition'(v) = w} \properadW\profilev \right] \times \left[\prod_{\partition'(v) = b} \properadB\profilev \right]  \times \left[\prod_{\partition'(v)=g} \properadG\profilev \right].
\end{multline*}
\item 
On gsd maps $f: (G,\partition) \to (J,\partition')$
with $G=J\{ K_v \}_{v\in \vertex(J)}$,
\[ D_{\dch}f: D_{\dch}(G,\partition) \to D_{\dch}(J,\partition') \]
is defined by 
\begin{multline*}
\left[\prod_{\partition(v') = w} \properadW( \inp(v') ; \out(v') ) \right] \times \left[\prod_{\partition(v') = b} \properadB( \inp(v') ; \out(v') ) \right]  \times \left[\prod_{\partition(v') = g} \properadG( \inp(v') ; \out(v') ) \right] \\
\overset{\cong}\to \left[\prod_{\partition'(v) = w} \properadW[K_v] \right] \times \left[\prod_{\partition'(v) = b} \properadB[K_v] \right]  \times \left[\prod_{\partition'(v) = g} \properadG[K_v] \right] \\
\overset{\gamma^{\properadW} \times \gamma^{\properadB} \times \gamma^{\properadG}}\to 
\left[\prod_{\partition'(v) = w} \properadW( \inp(K_v) ; \out(K_v) ) \right] \times \left[\prod_{\partition'(v) = b} \properadB( \inp(K_v) ; \out(K_v) ) \right]  \times \left[\prod_{\partition'(v) = g} \properadG( \inp(K_v) ; \out(K_v) ) \right]
\end{multline*}
where 
\[
	\gamma^{\properadW} : \properadW[K] \cong \prod_{v\in K} \properadW \profilev \to \properadW ( \inp(K) ; \out(K) )
\]
is the propic composition in $\properadW$, and similarly for $\gamma^{\properadB}$ and $\gamma^{\properadG}$.
\end{enumerate} 
\end{definition} 

\begin{proposition}
	$D_{\dch}:\blood\dc\rightarrow\sset$ is a functor.
\end{proposition}
\begin{proof}
	Notice that $D_{\dch}$ respects compositions when restricted to the subcategory containing only transfusions.
	Furthermore, when restricted to the subcategory containing only gsd maps, $D_{\dch}$ again respects compositions because $\properadW$, $\properadB$, $\properadG$ are props.
	To finish the proof, we must show that 
	\[\begin{tikzcd}
	D(G,\partition) \dar{D(f)} \rar{D(f^*q)} &  D(G,f^*\partition')\dar{D(f)}\\
	D(J,\partition'') \rar{D(q)} & D(J,\partition')
	\end{tikzcd}
	\]
	commutes, which follows from the fact that $m_w$ and $m_b$ are maps of props.
\end{proof}

\subsection{Colimits and Reflexive Coequalizers} In order to show that $\blood$ models the pushout \[
\begin{tikzcd} \properadG \dar{m_w} \rar{m_b} & \properadB \\
\properadW
\end{tikzcd}
\] we will make use of the following basic facts about reflexive coequalizers. Recall that a \emph{reflexive coequalizer} is a functor out of $\sk_1 \Delta^\op$.

\newsavebox\rcoex
\begin{lrbox}{\rcoex}
\begin{tikzcd}
X_0 \arrow{r}[description]{s_0} & X_1 
	\lar[to path={
		([yshift=1ex]\tikztostart.west)--([yshift=1ex]\tikztotarget.east) \tikztonodes
	}, swap]{d_0}
	\lar[to path={
		([yshift=-1ex]\tikztostart.west)--([yshift=-1ex]\tikztotarget.east) \tikztonodes
	}]{d_1}
\end{tikzcd}
\end{lrbox}

\begin{lemma}\label{reflexive_coequalizer}
Let $X_0, X_1$ be two props and $U:\Prop \to \Mega$ be the forgetful functor.
If
\[ \usebox\rcoex \]
is a reflexive coequalizer (i.e.\ $d_0s_0 = d_1s_0 = \id_{X_0}$) then
\[
	U\colim \left( \usebox\rcoex \right) = \colim \left( 
\begin{tikzcd}
UX_0 \arrow{r}[description]{Us_0} & UX_1 
	\lar[to path={
		([yshift=1ex]\tikztostart.west)--([yshift=1ex]\tikztotarget.east) \tikztonodes
	}, swap]{Ud_0}
	\lar[to path={
		([yshift=-1ex]\tikztostart.west)--([yshift=-1ex]\tikztotarget.east) \tikztonodes
	}]{Ud_1}
\end{tikzcd}
	\right).
\]
\end{lemma}

\begin{proof}
This follows from \cite[4.3]{RezkSA}.
\end{proof}

\begin{remark} 
We will apply Lemma~\ref{reflexive_coequalizer} to the following case. 
	If $\msp$ is an prop, we then have $\epsilon_{\msp}: \bot \msp \to \msp$ which is the adjoint of the identity $U\msp \to U\msp$.
	Furthermore, for any prop $\msp$ we have a map $\sigma_{\msp}$
	\begin{align*}
		\bot \msp &= F\id_{\Mega}U\msp \\
		&\overset\eta\to F\top U \msp \\
		&= FUFU\msp = \bot^2 \msp
	\end{align*}
	where $\eta: \id_{\Mega} \Rightarrow \top$ is the unit for the monad $\top$. 
	We then have a reflexive coequalizer diagram with $X_0 = \bot \msp$, $X_1 = \bot^2 \msp$, $d_0 =  \epsilon_{\bot \msp}$, $d_1 = \bot \epsilon_{\msp}$, and $s_0 = \sigma_{\msp}$.
	The colimit of this reflexive coequalizer is just $\msp$.

\end{remark}

\begin{lemma}\label{colim interchange lemma}
Let $m_w: \properadG \to \properadW$ and $m_b: \properadG \to \properadB$ be prop maps. Then 
\[ \colim \left( 
	\begin{tikzcd}
\bot\properadB \amalg_{\bot\properadG} \bot\properadW \arrow{r}[description]{s_0} & \bot^2\properadB \amalg_{\bot^2\properadG} \bot^2\properadW
	\lar[to path={
		([yshift=1ex]\tikztostart.west)--([yshift=1ex]\tikztotarget.east) \tikztonodes
	}, swap]{d_0}
	\lar[to path={
		([yshift=-1ex]\tikztostart.west)--([yshift=-1ex]\tikztotarget.east) \tikztonodes
	}]{d_1}
\end{tikzcd} \right) = \properadB \amalg_{\properadG} \properadW. \]
\end{lemma}

\begin{proof}
	Interchange of colimits on the product category \[ (\sk_1 \Delta)^\op \times (\bullet \leftarrow \bullet \rightarrow \bullet).\]
\end{proof}

The proof of the following proposition is elementary.
\begin{proposition}\label{proposition on pushouts}
	Suppose that we have a pushout diagram of sets
	\[ \begin{tikzcd}
	A \rar{e} \arrow[hook]{d}{f} & B \dar{g} \\
	C \rar{h} & D
	\end{tikzcd} \]
	with $f$ a monomorphism.
	Then $g$ is also a monomorphism and
	\[
		D = g(B) \amalg h(C \setminus f(A)) \cong B \amalg (C \setminus A).
	\]
	If $x,y \in C$ have $h(x) = h(y)$, then there are (unique) $\tilde x, \tilde y \in A$ with $x=f(\tilde x)$, $y = f (\tilde y)$ and $e(\tilde x) = e(\tilde y)$. \qed
\end{proposition}

For the purposes of the next lemma, fix a graph $G\in \sgc\dc$, let $\transfusions_G \subset \blood \dc$ be the subcategory whose objects are $(G,\partition)$ where $\partition$ ranges over all partitions of $\vertex(G)$ and whose morphisms are precisely the transfusions.

\begin{lemma}\label{alternate description for free guy}
	Suppose that $m_w: \properadG \to \properadW$ and
		$m_b: \properadG \to \properadB$
	are morphisms of props, that $m_w$ is a monomorphism of megagraphs, and that $D_{\dch}$ is the associated functor on $\blood\dc$. Let 
	\[
	\X = U\properadB \amalg_{U\properadG} U\properadW
	\]
	be the pushout in $\Mega$.
	Fix a graph $G\in \sgc\dc$ and consider the restriction of $D_{\dch}$ to the subcategory $\transfusions_G$.
	Then
	\[
		\colim\limits_{\transfusions_G} D_{\dch} \cong \X[G].
	\]
\end{lemma}
\begin{proof}
	For any partition $\partition$ of $G$, there is an obvious map $\alpha_\partition$ from $D_{\dch} (G,\partition)$ 
	
	to 
	\begin{align*}
	\X[G] &= 
	\prod_{v\in \vertex(G)} \X\profilev \\ &= 
	\prod_{v\in \vertex(G)} \properadB\profilev \uamalg{\properadG \profilev} \properadW \profilev. \end{align*}
	If $\partition \to \partition'$ is a transfusion, then the diagram
	\[ \begin{tikzcd}
	D_{\dch} (G,\partition) \arrow{dd} \arrow{dr}{\alpha_\partition} \\ & \X[G] \\ D_{\dch} (G,\partition') \arrow{ur}{\alpha_{\partition'}}
	\end{tikzcd} \]
	commutes, hence there exists a map
	\[
		\alpha: \colim\limits_{\transfusions_G} D_{\dch} \to \X[G].
	\]

	To see that $\alpha$ is a surjection, let $x \in \X[G]$, and, for each vertex $v\in \vertex(G)$, pick $y_v$ above $\pi_v(x)$
	\[ \begin{tikzcd}
	y_v \arrow[maps to]{d} & \in & \properadB\profilev \amalg \properadW \profilev \arrow[two heads]{d} \\
	\pi_v(x) & \in & \properadB\profilev \uamalg{\properadG \profilev} \properadW \profilev;
	\end{tikzcd} \]
	this choice of $y_v$ gives a partition of $G$ with
	\[
		\partition(v) = \begin{cases}
			w & y_v \in \properadW \profilev \\
			b & y_v \in \properadB \profilev
		\end{cases}
	\]
	By construction, the element $y = \smallprod{y_v} \in D(G,\partition)$ satisfies $\alpha_\partition(y) = x$. Hence $\alpha$ is a surjection.

	Injectivity is more subtle.
	Suppose that $x = \smallprod{x_v} \in D(G,\partition^x)_n$, $y = \smallprod{y_v} \in D(G,\partition^y)_n$, and $\alpha_{\partition^x}(x) = \alpha_{\partition^y}(y)$.
	We must show that $x$ and $y$ represent the same element in $\colim\limits_{\transfusions_G} D_{\dch}$.

	Fix a vertex $v$ and 
	write 
	\[
		A^b = \properadB\profilev \qquad A^g = \properadG\profilev \qquad A^w = \properadW\profilev
	\]
	for the relevant simplicial sets; now we have
	\[
		x_v \in A^{\partition^x(v)} \qquad \text{ and } \qquad y_v \in A^{\partition^y(v)}. 
	\]

	We will write
	\begin{align*}
		\pi_v^{\X}: \X[G] \to \X\profilev &= \properadB\profilev \uamalg{\properadG \profilev} \properadW \profilev \\
		&= A^b \amalg_{A^g} A^w;
	\end{align*}
	universally writing $m$ for the map $A^\square \to A^b \amalg_{A^g} A^w$ we have $\pi_v^{\X} \alpha_{\partition^x}(x) = mx_v$.
	The equality $\alpha_{\partition^x}(x) = \alpha_{\partition^y}(y)$ is equivalent to
	\[
		\pi_v^{\X} \alpha_{\partition^x}(x) = \pi_v^{\X} \alpha_{\partition^y}(y)
	\]
	for all $v \in \vertex(G)$.

	Based on $x$ and $y$, we create an interpolating partition $\partition^g$ equipped with transfusions $\partition^g \to \partition^x$ and $\partition^g \to \partition^y$. We simultaneously construct elements $x'=\smallprod{x_v'} , y'=\smallprod{y_v'} \in D(G,\partition^g)$ satisfying $mx_v' = my_v'$. Finally, we construct a fourth partition $\partition^b$ which admits a black transfusion $\partition^g \to \partition^b$, so that we have the following schematic.
\[ \begin{tikzcd}
&&& z \arrow[squiggly]{dl}{\in} \\
x \dar[squiggly]{\in} & x' \dar[squiggly]{\in} \lar[maps to] \arrow[maps to, bend left]{urr} & D(G,\partition^b) \\
D(G,\partition^x) & D(G,\partition^g) \lar \dar \arrow{ur} & y' \dar[maps to] \arrow[maps to, bend right]{uur} \lar[squiggly]{\in} \\
 & D(G,\partition^y) &y \lar[squiggly]{\in}
\end{tikzcd} \]

	Fix a vertex $v$; we will specify a color for $v$ in the partitions $\partition^g, \partition^b$ and also define the components $x_v'$, $y_v'$.
\begin{itemize}
	\item Suppose $\partition^x(v) = \partition^y(v)$ and $x_v = y_v$.  Then set 
	$\partition^b (v) = \partition^g(v) = \partition^x(v)$ and $x_v' = y_v' = x_v$. 
	\item Suppose $\partition^x(v) = \partition^y(v)$ and $x_v \neq y_v$. 
	Set $\partition^g(v) = g$ and $\partition^b(v) = b$.
\begin{itemize}
	\item We have $\partition^x(v) \neq b$ since \[ A^b \hookrightarrow A^b \amalg_{A^g} A^w \]
	is a monomorphism (Proposition~\ref{proposition on pushouts}) and $mx_v = \pi_v^{\X}(\alpha_{\partition^x} x) = \pi_v^{\X}(\alpha_{\partition^y}y) = my_v$.
	\item If $\partition^x(v) = \partition^y(v) = w$, then use Proposition~\ref{proposition on pushouts} to find elements $x'_v, y'_v \in A^g$ which map to $x_v, y_v \in A^w$. Notice that $m_bx_v' = m_by_v'$.
	\item If $\partition^x(v) = \partition^y(v) = g$, then set $x'_v = x_v$, $y'_v = y_v \in A^g$. Notice that these elements have a common image in $A^b$.
\end{itemize}
	\item If $\partition^x(v) \neq \partition^y(v)$, then set $\partition^g(v) = g$ and $\partition^b(v) = b$. There are six cases, but they are symmetric; reverse the roles of $x_v$ and $y_v$ when necessary.
	\begin{itemize}
		\item If $\partition^x(v) = g$ and $\partition^y(v) = b$, then since $mx_v = my_v$ we actually have $m_bx_v = y_v$.
		Set $x'_v = y'_v = x_v$. 
		\item If $\partition^x(v) = g$ and $\partition^y(v) = w$, then set $x_v'= x_v$. We have $m_wx_v = my_v$, so there is an element $y_v' \in A^g$ with $m_wy_v' = y_v$, and $m_bx_v' = m_by_v'$. 
		\item If $\partition^x(v) = w$ and $\partition^y(v) = b$, then $x_v = m_w(x_v')$ for some $x_v'$, and, furthermore, $m_b(x_v') = y_v$. Set $y_v' = x_v'$.
	\end{itemize}
\end{itemize}
By construction, $x'$ and $y'$ map to the same element of $D(G,\partition^b)$, hence $x$ and $y$ represent the same element of 
$\colim\limits_{\transfusions_G} D_{\dch}$.

\end{proof}

\begin{theorem}\label{pushout vs colim}
	Suppose that $m_w: \properadG \to \properadW$ and
		$m_b: \properadG \to \properadB$
	are morphisms of props, that $m_w$ is a monomorphism of megagraphs, and that $D_{\dch}$ is the associated functor on $\blood\dc$. 
	Consider the pushout
	\[ \begin{tikzcd}
	\properadG \rar{m_b} \dar{m_w} & \properadB \dar \\  \properadW \rar & \properadB \amalg_{\properadG} \properadW.
	\end{tikzcd} \]
	Then 
	\[
		\left(\properadB \amalg_{\properadG} \properadW\right)\dc \cong \colim\limits_{\blood\dc} D_{\dch}.
	\]
\end{theorem}

\begin{proof}
	By Lemmas \ref{reflexive_coequalizer} and \ref{colim interchange lemma}, it is enough to compute
\[ 	\colim \left( 
	\begin{tikzcd}
U\left[ \bot\properadB \amalg_{\bot\properadG} \bot\properadW \right] \arrow{r}[description]{Us_0} & U\left[\bot^2\properadB \amalg_{\bot^2\properadG} \bot^2\properadW \right]
	\lar[to path={
		([yshift=1ex]\tikztostart.west)--([yshift=1ex]\tikztotarget.east) \tikztonodes
	}, swap]{Ud_0}
	\lar[to path={
		([yshift=-1ex]\tikztostart.west)--([yshift=-1ex]\tikztotarget.east) \tikztonodes
	}]{Ud_1}
\end{tikzcd} \right)\]
at a profile $\dch$.
Colimits in $\Mega$ are computed levelwise, so we're really after
\[ 	\colim \left( 
	\begin{tikzcd}
\left[ \bot\properadB \amalg_{\bot\properadG} \bot\properadW \right]\dc \arrow{r}[description]{s_0} & \left[\bot^2\properadB \amalg_{\bot^2\properadG} \bot^2\properadW \right]\dc
	\lar[to path={
		([yshift=1ex]\tikztostart.west)--([yshift=1ex]\tikztotarget.east) \tikztonodes
	}, swap]{d_0}
	\lar[to path={
		([yshift=-1ex]\tikztostart.west)--([yshift=-1ex]\tikztotarget.east) \tikztonodes
	}]{d_1}
\end{tikzcd} \right).\]
For notational convenience we will write
\begin{align*}
	\X &= U\properadB \amalg_{U\properadG} U\properadW, \\
	\Y &= U\bot \properadB \amalg_{U\bot \properadG} U\bot \properadW;
\end{align*}
since left adjoints commute with colimits we have
\begin{align*}
	\bot\properadB \amalg_{\bot\properadG} \bot\properadW &= F\left( U\properadB \amalg_{U\properadG} U\properadW \right) = F\X \\
	\bot^2\properadB \amalg_{\bot^2\properadG} \bot^2\properadW &= F\Y.
\end{align*}

Let $I\in \sgc\dc$ be a set of representatives of isomorphism classes as in Proposition \ref{free prop description}. 
For each $(G,\partition) \in \blood\dc$, there is an $H\in I$ which is isomorphic to $G$, and we have
\[
	D_{\dch}(G,\partition) \overset{\cong}{\to} D_{\dch} (H, \partition') \to \X[H] \subset F\X\dc.
\]
The composites
\[
	D_{\dch}(G,\partition) \to F \X \dc \to \left( \properadB \amalg_{\properadG} \properadW \right)\dc
\]
are compatible with maps in $\blood\dc$ and induce
\begin{equation} \label{first map}
	\colim_{\blood\dc} D_{\dch} \to \left( \properadB \amalg_{\properadG} \properadW \right)\dc.
\end{equation}

The diagram
\[ \begin{tikzcd}
\coprod_{G\in I} \colim\limits_{\transfusions_G} D_{\dch}(G,\partition) \rar{\cong} \dar &
\coprod_{G\in I} \X[G] \cong F\X \dc \\
\colim\limits_{\blood\dc} D_{\dch}(G,\partition)
\end{tikzcd} \]
shows that we have a map
\[
	p: F\X\dc \to \colim\limits_{\blood\dc} D_{\dch}(G,\partition);
\]
we need to show that this induces a map out of 
\[ 	\properadB \amalg_{\properadG} \properadW \dc = \colim \left( 
	\begin{tikzcd}
F \X \dc \arrow{r}[description]{s_0} & F \Y \dc
	\lar[to path={
		([yshift=1ex]\tikztostart.west)--([yshift=1ex]\tikztotarget.east) \tikztonodes
	}, swap]{d_0}
	\lar[to path={
		([yshift=-1ex]\tikztostart.west)--([yshift=-1ex]\tikztotarget.east) \tikztonodes
	}]{d_1}
\end{tikzcd} \right).\]
It is enough to show that the two composites
\begin{align}
F \Y \dc \overset{d_0}\longrightarrow F \X \dc &\overset{p}\to  \colim\limits_{\blood\dc} D_{\dch}(G,\partition) \label{d0composite}\\
F \Y \dc \overset{d_1}\longrightarrow F \X \dc &\overset{p}\to  \colim\limits_{\blood\dc} D_{\dch}(G,\partition)\label{d1composite}
\end{align}
are equal.

Write
\[
	D_{\dch}^\bot: \blood\dc \to \sset
\]
for the functor associated to the prop maps
\[
	\bot \properadB \overset{\bot m_b}{\longleftarrow} \bot \properadG \overset{\bot m_w}{\longrightarrow} \bot \properadW.
\]
Lemma~\ref{alternate description for free guy} applied to $\bot m_w$ and $\bot m_b$ tells us that the natural map
	\begin{equation}\label{colimsY}
		\colim\limits_{\transfusions_G} D_{\dch}^\bot \overset{\cong}{\to} \Y[G].
	\end{equation}
is an isomorphism.

Fix a graph $G\in I \subset \sgc\dc$, and let $\tilde y \in \Y[G]$.
Pick a representative $y\in D_{\dch}^\bot(G,\partition)$ for $\tilde y$ using \eqref{colimsY} with $\partition^{-1}(g) = \varnothing$.

For each $v\in \vertex(G)$, let $I_v \subset \sgc\profilev$ be a set of representatives for isomorphism classes.
We have
\begin{align*}
D_{\dch}^\bot(G,\partition) &= \prod_{\partition(v) = b} FU\properadB \profilev \times \prod_{\partition(v) = w} FU\properadW \profilev \\
&= \prod_{\partition(v) = b} \left(\coprod_{K_v \in I_v} \properadB [K_v]\right) \times \prod_{\partition(v) = w} \left(\coprod_{K_v \in I_v} \properadW [K_v]\right) \\
&= \coprod_{\prod\limits_v I_v} \left( \prod_{\partition(v) = b}  \properadB [K_v] \times \prod_{\partition(v) = w} \properadW [K_v] \right)
\end{align*}
so 
\[
	y\in \left( \prod_{\partition(v) = b}  \properadB [K_v] \times \prod_{\partition(v) = w} \properadW [K_v] \right)
\]
for some choice of graphs $K_v$, one for each vertex of $G$.
Let $\partition'$ be the partition of $G\{K_v\}$ so that each vertex of $K_v$ has the same color as $v$.

We have the following commutative diagram
\[ \begin{tikzcd}[column sep=0ex]
y \arrow[squiggly]{d}{\in} \arrow[squiggly]{r}{\in}& D_{\dch}^\bot(G,\partition) \dar{=} \\
\prod\limits_{\partition(v) = b}  \properadB [K_v] \times \prod\limits_{\partition(v) = w} \properadW [K_v]  \arrow[hook]{r} \arrow{dd}{=} & 
\coprod\limits_{\prod\limits_v I_v} \left( \prod\limits_{\partition(v) = b}  \properadB [K_v] \times \prod\limits_{\partition(v) = w} \properadW [K_v] \right) \dar{\coprod (\prod \gamma_\properadB \times \prod \gamma_\properadW)} \\
&   \prod\limits_{\partition(v) = b}  \properadB \profilev \times \prod\limits_{\partition(v) = w} \properadW \profilev \dar{=} \\
D(G\{K_v\}, \partition') \arrow{r} \arrow{dr} & D(G,\partition) \arrow{d} \\
&\colim\limits_{\blood\dc} D_{\dch}(G,\partition).
\end{tikzcd} \]
Composition along the left is \eqref{d0composite} while composition along the right is \eqref{d1composite}.
Hence $pd_0(\tilde y) = pd_1(\tilde y)$. Since $\tilde y$ was arbitrary, $pd_0 = pd_1$ and we have that $p$ induces 
\[
	\properadB \amalg_{\properadG} \properadW \dc \to \colim\limits_{\blood\dc} D_{\dch}(G,\partition)
\]
which is the inverse to \eqref{first map}.
\end{proof}

\section{Specialization to pushouts along \texorpdfstring{$\msi \to \msh$}{I --> H}}\label{section specialization}

At this point we wish to specialize to pushouts where $\properadG \to \properadW$ is one of the maps $\msi \to \msh$.
To do so, we will look at a subcategory of $\blood\dch$ called $\bloodsimple\dch$ whose objects are \emph{simplified} graphs; this subcategory has the property that $\colim_{\bloodsimple\dc} D = \colim_{\blood\dc} D$ (Theorem~\ref{equal_colims}).
First we will reduce to an intermediary class of graphs: the admissible graphs.

\begin{definition}\label{bloodIHdef} Fix an input-output profile $\dch$.
\begin{itemize}
	\item An object $(G,\partition) \in \blood\dch$ will be called \emph{admissible} if
	\begin{itemize}
		\item For each white vertex $v$, 
		\[ \profilev \in \left\{ \xxsingle, \yxsingle, \xysingle, \yysingle \right\},\]
		\item For each grey vertex $v$, $\profilev = \xxsingle$, and
		\item For each black vertex $v$, $y\notin \inp(v) \cup \out(v)$.
	\end{itemize}
	\item For a fixed input-output profile, $\dch$, let $\bloodadmissible\dch$ be the 
	full subcategory of $\blood\dch$ 
	consisting of those $(G,\partition)$ which are admissible.
	\item For two graphs $G,K$, a \emph{subgraph structure} $K \leq G$ is a graph $J$ together with a vertex $v_0 \in \vertex(J)$ so that $G = J(K)$, with $K$ substituted at the vertex $v_0$.
\end{itemize}
\end{definition}

The functor $F_0 : \sCat \to \sProp$ takes a simplicial category $\mathcal{C}$ to the simplicial prop with morphism spaces defined by 
\begin{equation*}
	F_0\mathcal{C}(a_1, \dots, a_n; b_1, \dots, b_m) = \begin{cases} \coprod\limits_{\theta} \prod\limits_{i=1}^n \mathcal{C}(a_{\theta^{-1}(i)}, b_i) & n =m \geq 0 \\
	\varnothing & \text{ otherwise} \end{cases}
\end{equation*}
where $\theta$ ranges over bijections $\theta: \{1, \dots, n\} \to \{1, \dots, n\}$ (see \cite[Proposition 11]{hackneyrobertson1}).
This implies that if $D(G,\partition) \neq \varnothing$, then for each white or grey vertex $v$ we must have $|\inp(v)| = |\out(v)| = n \geq 0$.
Given such a white or grey vertex $v$, 
we can `blow-up' the graph $G$ in $n!$ ways, with the vertex $v$ replaced by $n$ vertices\footnote{If $n=0$, this blow-up operation is just deletion of the isolated vertex $v$, while if $n=1$ the result is the same graph $G$.} (of the same color as $v$).
This gives $n!$ objects $(G_\theta, \partition_\theta)$, indexed on bijections $\theta: \{1, \dots, n\} \to \{1, \dots, n\}$ so that the gsd maps give an isomorphism
\[
	\coprod_{\theta} D(G_\theta, \partition_\theta) \overset\cong\to D(G,\partition).
\]
For example, if $n=2$ then the $2! = 2$ modifications of a vertex with two inputs and two outputs are as in figure~\ref{vertex expansion figure}; the rest of the graph is left untouched.

\begin{figure}
\includegraphics[width=0.6\textwidth]{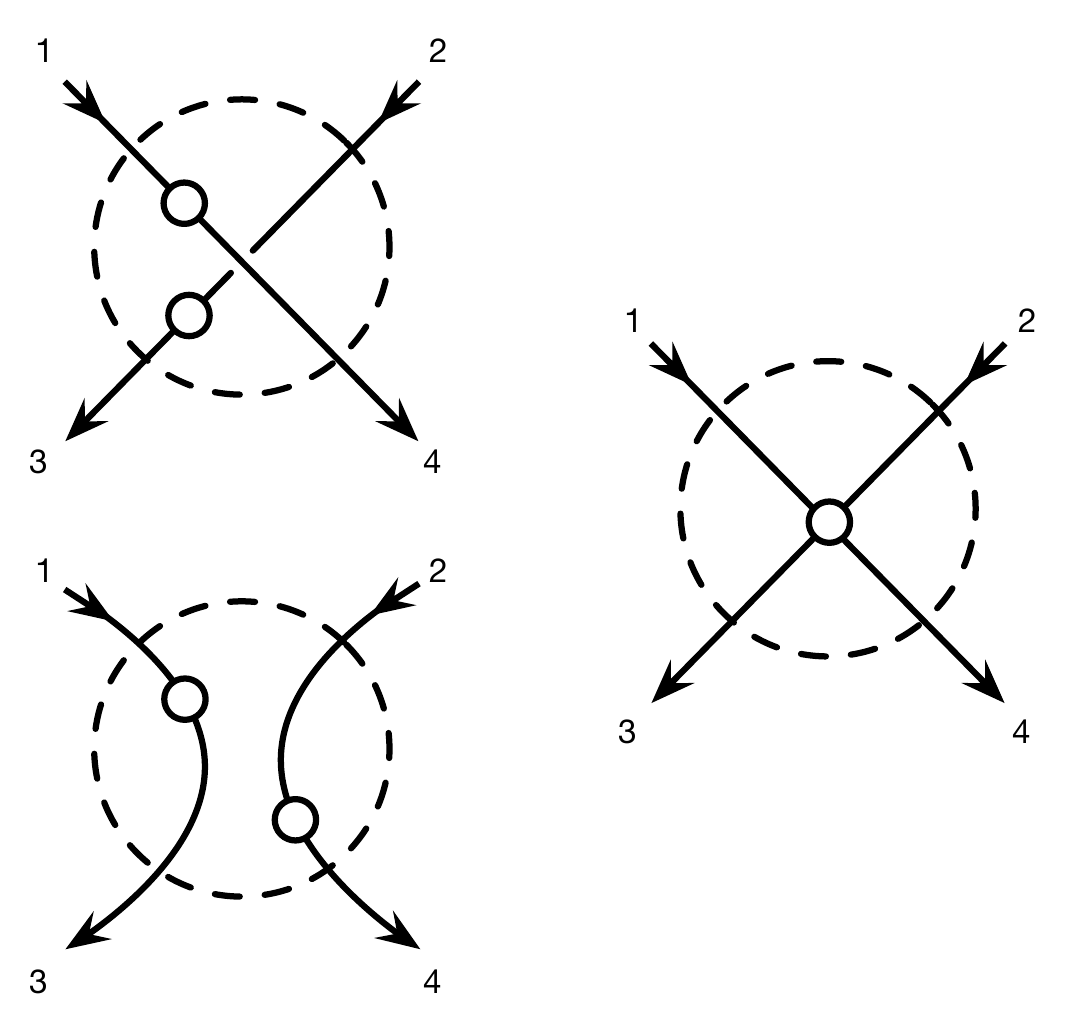}
\caption{Blow-up at a 2-2 corolla}\label{vertex expansion figure}
\end{figure}

The \emph{vertex expansion} of $(G,\partition)$ is obtained by iterated blow-ups of all white and grey vertices.
We have a collection of gsd maps with
\begin{itemize}
	\item $(G_\alpha, \partition_\alpha) \to (G,\partition)$ is a gsd map 
\[ G_\alpha = G\{ K_v \}_{\partition(v) \in \{ w,g\} } \]
	\item each $(G_\alpha, \partition_\alpha)$ is admissible in the sense of Definition~\ref{bloodIHdef}, 
	\item each $K_v$ is a (possibly empty) graph whose connected components are linear graphs with at least one vertex, and
	\item the map 
	\begin{equation}\label{e:vertex expansion}
	\coprod_{\alpha} D(G_\alpha, \partition_\alpha) \overset\cong\to D(G,\partition).
\end{equation}
is an isomorphism.
\end{itemize}
We call \eqref{e:vertex expansion} the vertex expansion; note that it is the identity if $(G,\partition)$ is already admissible.

\begin{definition} \emph{Atomic morphisms} of $\blood\dch$ are one of the following:
	\begin{itemize}
		\item An \emph{atomic transfusion} changes precisely one grey vertex to either a white or black vertex. We call these white or black atomic transfusions, respectively.
		\item An \emph{atomic gsd map} is a gsd map of the form
		\[
			(G(K),\partition') \to (G,\partition)
		\]
		given by doing a graph substitution at a single vertex $v\in \vertex(G)$. Depending on $\partition(v)$, we say this is a white, grey, or black atomic gsd map.
	\end{itemize}
\end{definition}

\begin{lemma}\label{lemma_vertex_expansion}
Suppose that $(G,\partition) \to (J,\partition)$ is a map in $\blood\dc$ with $D(G,\partition) \neq \varnothing$, and that
\begin{align*}
	\coprod_{\alpha} D(G_\alpha, \partition_\alpha) &\overset\cong\to D(G,\partition) &
	\coprod_{\beta} D(J_\beta, \partition'_\beta) &\overset\cong\to D(J,\partition')
\end{align*}
are the vertex expansions \eqref{e:vertex expansion}.
Then, for each $\alpha_0$ there is a morphism
$(G_{\alpha_0}, \partition_{\alpha_0}) \to (J_{\beta_0}, \partition'_{\beta_0})$ making the diagram
\[ \begin{tikzcd}
	D(G_{\alpha_0}, \partition_{\alpha_0}) \rar \dar & \coprod_{\alpha} D(G_\alpha, \partition_\alpha) \rar{\cong} & D(G,\partition) \dar \\
D(J_{\beta_0}, \partition'_{\beta_0}) \rar &
	\coprod_{\beta} D(J_\beta, \partition'_\beta) \rar{\cong} & D(J,\partition')
\end{tikzcd} \]
commute.
\end{lemma}
\begin{proof}
Write $G_{\alpha_0} = G\{H_v\}_{v\in \vertex(G)}$, with $H_v$ a corolla for every black vertex $v$ and $H_v$ has (1,1) corollas as its connected components for each grey or white vertex $v$.
It is enough to prove the lemma when $(G,\partition) \to (J,\partition)$ is an atomic morphism.
The most trivial case is that of a black atomic gsd map (which we omit), white and grey atomic gsd maps share the same argument, and black atomic transfusions are interesting.

\begin{figure}
	\begin{tikzpicture}
		\node(blownup){\usebox\threeblowups};
		\node[right=3cm of blownup](orig){\usebox\toblowup};
		\path[draw,->] (blownup) -- (orig);
		\node[below=of blownup](erasedbb){\usebox\erasedboundingbox};
		\node[below=of erasedbb](collapsed){\usebox\collapsedvertices};
		\node[right=3cm of collapsed](corol){\usebox\corollafives};
		\draw[double equal sign distance] (blownup) to node[left] {$K\{H_{v}\}$} (erasedbb);
		\path[draw,->] (erasedbb) -- (collapsed);
		\path[draw,->] (orig) -- (corol);
		\path[draw,->] (collapsed) -- (corol);
		\node[right=1mm of orig](k){$K$};
		\node[right=-9mm of corol](v){$v_0$};
	\end{tikzpicture}
	\caption{White or grey subgraph contraction}\label{subgraph_contr}
\end{figure}

We begin with the case when we have a grey or white atomic gsd map.
Suppose that $G=J(K)$, with $K$ substituted into a grey or white vertex $v_0$ of $J$.
Then the graph $K\{H_v\}_{v\in \vertex(K)}$ is a subgraph of $G_{\alpha_0}$ whose connected components are again linear graphs containing at least one vertex. 
\[ G_{\alpha_0} = G\{H_w\}_{w\in \vertex(G)} = J(K)\{H_v\}_{v\in \vertex(G)} = (J\{H_v\}_{v\in \vertex(G)\setminus \vertex(K)})(K\{H_v\}_{v\in \vertex(K)})\]
Contracting each component of $K\{H_v\}_{v\in \vertex(K)}$ to a (1,1)-corolla gives a blow-up of the vertex $v_0$ in $(J,\partition)$. See Figure~\ref{subgraph_contr}.

For the case of a white atomic transfusion, suppose that $v_0$ is transfused from grey to white. Then the transfusion $\partition_{\alpha_0} \to \partition'_{\alpha_0}$ which turns all vertices of $H_{v_0}$ into white vertices gives the commutative diagram
\[ \begin{tikzcd}
	(G_{\alpha_0}, \partition_{\alpha_0}) \rar \dar & (G,\partition) \dar \\
(G_{\alpha_0}, \partition_{\alpha_0}') \rar & (J,\partition'),
\end{tikzcd} \]
and $(G_{\alpha_0}, \partition_{\alpha_0}') \to (J,\partition')$ is a vertex expansion of $(J,\partition')$.

Finally, for case of a black atomic transfusion where we transfuse $v_0$ from grey to black, we first transfuse all vertices of $H_{v_0}$ to black vertices. We get something that is no longer a vertex expansion of $J$ since a black vertex has been blown up. So we follow this by a contraction of the subgraph $H_{v_0}$ to get a vertex expansion of $J$. See Figure~\ref{black_tr}.
\begin{figure}
\begin{tikzpicture}
\node(greydude){\usebox\greytangle};
\node[below=of greydude](blackdude){\usebox\blacktangle};
\node[below=of blackdude](cbv){\usebox\corollablackvariant};
\path[draw,->] (greydude) -- (blackdude);
\path[draw,->] (blackdude) -- (cbv);
\node[right=3cm of cbv](cb){\usebox\corollathreesblack};
\node[right=3cm of greydude](cg){\usebox\corollathreesgrey};
\path[draw,->] (greydude) -- (cg);
\path[draw,->] (cbv) -- (cb);
\path[draw,->] (cg) -- (cb);
\node[left=-1mm of greydude](hv){$H_{v_0}$};
\node[right=-5mm of cg](v){$v_0$};
\end{tikzpicture}
	\caption{Black transfusion}\label{black_tr}
\end{figure}

\end{proof}

\begin{proposition}\label{equal_colims1}
Let $D_{\dch}=D$ be the functor $\blood\dc \to \sset$ determined by the maps $ \msi \to  \msh$ and $ \msi \to \msp$. Then the map
	\[
		\colim_{\bloodadmissible\dc} D \to \colim_{\blood\dc} D
\] 
is an isomorphism.
\end{proposition}

\begin{proof}
We construct an inverse. For each $(G,\partition) \in \blood\dc$ with $D(G,\partition) \neq \varnothing$, we have the composite
	\[
		D(G,\partition) \overset\cong\leftarrow \coprod_{\alpha} D(G_\alpha, \partition_\alpha) \to \colim_{\bloodadmissible\dc} D
	\]
	where the map on the left is vertex expansion \eqref{e:vertex expansion}; we wish to show that this induces a map from $\colim_{\blood\dch} D$.
	If $(G,\partition) \to (J,\partition')$ is a morphism of $\blood\dch$, then for each $\alpha_0$ we have the commutative diagram
		\[ \begin{tikzcd}
	D(G,\partition) \rar & 
		D(J,\partition') \\
	\coprod_{\alpha} D(G_\alpha, \partition_\alpha) \uar{\cong} & 
		\coprod_{\beta} D(J_\beta, \partition'_\beta) \uar{\cong}[swap]{\eqref{e:vertex expansion}} \\
 	D(G_{\alpha_0}, \partition_{\alpha_0}) \uar \rar{\ref{lemma_vertex_expansion}} \dar & 
 		D(J_{\beta_0}, \partition'_{\beta_0})  \uar \dar \\
 	\colim\limits_{\bloodadmissible\dc} \decAB \rar{=} & 
 		\colim\limits_{\bloodadmissible\dc} \decAB
	\end{tikzcd} \]
	Summing over all $\alpha$ shows that we have a well-defined map 
	$\colim_{\blood \dch} D \to \colim_{\bloodadmissible \dch} D.$
	It is immediate that the composite
	\[
		\colim_{\bloodadmissible\dc} \decAB \to \colim_{\blood\dc} D \to \colim_{\bloodadmissible\dc} \decAB 
	\]
	is the identity.

	To show that 
	\[
		\colim_{\blood\dc} D \to \colim_{\bloodadmissible\dc} \decAB \to \colim_{\blood\dc} D
	\]
	is the identity, it is enough to show that the diagram
	\[ \begin{tikzcd}
	& & D(G,\partition) \arrow{dll}{a} \dar{a} \\
	\colim\limits_{\blood\dc} D\rar{g} & \colim\limits_{\bloodadmissible\dc} \decAB \rar{f} &\colim\limits_{\blood\dc}D
	\end{tikzcd} \]
	commutes for each object $(G,\partition) \in \blood\dch$.
	For an admissible $(G,\partition)$ we have
	\[ \begin{tikzcd}
	& & & D(G,\partition) \arrow{dll}{a} \dar{a} \arrow[dashed, bend right]{dlll}{a'} \\
	\colim\limits_{\bloodadmissible\dc} \decAB \rar{f} & 
	\colim\limits_{\blood\dc} D\rar{g} & 
	\colim\limits_{\bloodadmissible\dc} \decAB \rar{f} &
	\colim\limits_{\blood\dc} D
	\end{tikzcd} \]
	and we know that $gfa' = a'$, hence $fga = fgfa' = fa' = a$.
	But using the decomposition of $D(G,\partition)$ into the coproduct over $\coprod_{\alpha} D(G_\alpha, \partition_\alpha)$
	we have $fga=a$ for an arbitrary $(G,\partition)$.
\end{proof}

\subsection{Grey reduction of admissible graphs}

The following definition doesn't need the full notion of \emph{admissibility}, but it \emph{does} need that grey vertices have input-output profile $\xxsingle$.

\begin{definition}
	If $(G,\partition)\in\bloodadmissible\dc$, let $\eta(G,\partition) = (G^{\setminus \partition}, \partition^{\setminus g})$ be the object with
	\[
		G^{\setminus \partition} = G \{ \,\,|_x \,\}_{\partition(v) = g}
	\]
	and $\partition^{\setminus g} = \partition|_{\vertex(G) \setminus \partition^{-1}(g)}$.
	Let $gr(G,\partition)$ be the gsd map\[
		\eta(G,\partition) \to (G,\partition).
	\]
	We call $gr(G,\partition)$ or $\eta(G,\partition)$ the \emph{grey-reduction of} $(G,\partition)$.
\end{definition}

\begin{lemma}\label{grey reduction no change} If $D$ is the functor associated with 
\[  \msh \leftarrow  \msi \rightarrow \msp,\]
then
	\[
		D(gr(G,\partition)) : D(G^{\setminus \partition}, \partition^{\setminus g}) \to D(G,\partition)
	\]
	is an isomorphism.
\end{lemma}
\begin{proof}
	This is immediate since $F_0\msi \xxsingle = \msi(x,x) = *$.
\end{proof}

We wish to show that grey reduction is functorial.
In what follows, we will assume that \emph{all objects are admissible}.
Suppose that we have a transfusion
$q: (G,\partition') \to (G,\partition)$.
Then
\begin{equation}\label{e:gr transf}
\begin{aligned}
	G^{\setminus \partition'} &= G \{ \,\,|_x \,\}_{\partition'(v) = g} \\
	&= G \{ \,\,|_x \,\}_{\partition(v) = g} \{ \,\,|_x \,\}_{\substack{\partition'(v) = g \\ \partition(v) \neq g}} \\
	&= G^{\setminus \partition} \{ \,\,|_x \,\}_{\substack{\partition'(v) = g \\ \partition(v) \neq g}}
\end{aligned}\end{equation}
which determines a gsd map $\eta(q)$ which is the top map of the diagram
\[ gr(q) = \left( \begin{tikzcd} 
\eta(G,\partition') \dar \rar{\eta(q)} & \eta(G,\partition) \dar
\\
(G,\partition') \rar{q} & (G,\partition)
\end{tikzcd} \right). \]
Suppose that we have a gsd map
$f: (G,\partition) \to (H,\partition')$;
write \[ G = H\{K_v\}_{\partition'(v) \in \{b,w\}}\{ L_v \}_{\partition'(v) = g} \]
where each $L_v$ is connected and linear by admissibility.
Setting $G_0 = H\{ L_v \}_{\partition'(v) = g}$, we then have a factorization of $f$ as
\[
	(G,\partition) \to (G_0, \partition_0) \to (H,\partition')
\]
where $G = G_0 \{ K_v \}_{\partition'(v) \in \{b,w\}}$.
For the first map we have
\begin{equation}\label{e: G to G0}
\begin{aligned}
	G^{\setminus \partition} &= G\{ \,\,|_x \,\}_{\partition(v') = g} \\ 
	&= \left(G_0 \{ K_v \}_{\partition'(v) \in \{b,w\}}\right) \{ \,\,|_x \,\}_{\partition(v') = g} \\
	&= G_0 \{ K_v \}_{\partition_0(v) \in \{b,w\}} \{ \,\,|_x \,\}_{\partition_0(v) = g} \\
	&= G_0 \{ \,\,|_x \,\}_{\partition_0(v) = g} \{ K_v \}_{\partition_0(v) \in \{b,w\}} \\
	&= G_0^{\setminus \partition_0} \{ K_v \}_{\partition_0(v) \in \{b,w\}} 
\end{aligned}\end{equation}
hence get a gsd map
\[
	\eta(G,\partition) \to \eta(G_0, \partition_0).
\]
By admissibility, each $L_v$ is a connected, linear graph, and we have
\begin{equation}\label{e: gr G0 to H}
\begin{aligned}
	G_0^{\setminus \partition_0} &= G_0\{ \,\,|_x \,\}_{\partition_0(v') = g} \\ 
	&= \left( H \{ L_v \}_{\partition' (v) = g} \right) \{ \,\,|_x \,\}_{\partition_0(v') = g} \\
	&= H \{ L_v \{ \,\,|_x \,\}_{v' \in \vertex(L_v)} \}_{\partition' (v) = g} \\
	&= H \{ \,\,|_x \,\}_{\partition'(v) = g} = H^{\setminus \partition'}
\end{aligned}\end{equation}
and set
\[ gr(f) = \left( \begin{tikzcd}
\eta(G,\partition) \rar{gsd} \dar & \eta(G_0, \partition_0) \rar{=} & \eta(H,\partition') \dar \\
(G,\partition) \rar & (G_0, \partition_0) \rar & (H,\partition')
\end{tikzcd}\right). \]
It is now easy to check that 
\begin{align*}
	gr(f\circ f') &= gr(f) \circ gr(f') \\
	gr(q \circ q') &= gr(q) \circ gr(q')
\end{align*}
where $f,f'$ are gsd maps and $q,q'$ are transfusions.
We then \emph{define}
\[
	gr(f\circ q) = gr(f) \circ gr(q)
\]
whenever $f$ is a gsd map and $q$ is a transfusion.

\begin{proposition}\label{grey red functoriality} 
Grey reduction $gr: \bloodadmissible\dc \to Ar(\bloodadmissible\dc)$ is a functor.
\end{proposition}
\begin{proof}
	We need to show that
	\[
		gr(q\circ f) = gr(q) \circ gr(f)
	\]
	when $f$ is a gsd map and $q$ is a transfusion.
	Consider the situation
	\[ \begin{tikzcd}
		(G,\partition_2) \dar{f} \rar[color=red]{f^*q} &
	\color{red} (G,f^*\partition_3) \dar[color=red]{f} \\ 
	(J,\partition_2') \rar{q} & (J,\partition_3)
	\end{tikzcd} \]
	from \eqref{composition def} on page \pageref{composition def}.
	We must show that
	\[
		gr({\color{red} f}) \circ gr({\color{red} f^*q}) = gr(q) \circ gr(f).
	\]
	It is enough to check that the composites of gsd maps
	\begin{gather}
\label{g: G}		G^{\setminus \partition_2} \to G^{\setminus f^* \partition_3} \to J^{\setminus \partition_3}\\
\label{g: J}		G^{\setminus \partition_2} \to J^{\setminus \partition_2'} \to J^{\setminus \partition_3}
	\end{gather}
	are equal.

The map $f$ is given by a decomposition
\[
	G = J\{ K_v \}_{\partition_2'(v) \in \{ b,w \}}\{ M_v \}_{\substack{\partition_3(v) \in \{ b,w \}\\\partition_2'(v) =g}}\{ L_v \}_{\partition_3(v) =g}.
\]
By admissibility, 
each $M_v$ and $L_v$ is a connected linear graph with each edge colored by $x$.

Set $G_1 = J\{ L_v \}$ and $G_0 = J\{ M_v, L_v \} = G_1 \{ M_v \}$; then we have decompositions
\begin{align*}
f : (G&,\partition_2) \to (G_0, \partition_0) \to (J,\partition'_2) &
	{\color{red} f} : (G&,f^*\partition_3) \to (G_1, \partition_1) \to (J,\partition_3) \\ 
	G &= G_0\{ K_v \}_{\partition_2'(v) \in \{b,w\}}  & G &= G_1 \{ K_v, M_v \}_{\partition_3(v) \in \{b,w\}},
\end{align*}
which gives
\begin{align}
\label{a: gsd part2}
	G^{\setminus \partition_2} &\overset{\eqref{e: G to G0}}{=} G_0^{\setminus \partition_0} \{ K_v \}_{\partition_0(v) \in \{b,w\}} \overset{\eqref{e: gr G0 to H}}{=} J^{\setminus \partition_2'} \{ K_v \}_{\partition_2'(v) \in \{b,w\}}  \\
\label{a: gsd part3}
	G^{\setminus f^*\partition_3} 
	&\overset{\eqref{e: G to G0}}{=} G_1^{\setminus \partition_1} \{ K_v, M_v \}_{\partition_1(v) \in \{b,w\}}  
	\overset{\eqref{e: gr G0 to H}}{=} J^{\setminus \partition_3} \{ K_v, M_v \}_{\partition_3(v) \in \{b,w\}}.
\end{align}
Focusing now on the transfusions
$q: (J,\partition_2') \to (J,\partition_3)$ and $f^*q: (G,\partition_2) \to (G,f^*\partition_3)$, we have by \eqref{e:gr transf}
\begin{align}
\label{a: transf J}	J^{\setminus \partition_2'} 
	&= J^{\setminus \partition_3} \{ \,\,|_x \,\}_{\substack{\partition_2'(v) = g \\ \partition_3(v) \neq g}}\\
\label{a: transf G}	G^{\setminus \partition_2} 
	&= G^{\setminus f^*\partition_3} \{ \,\,|_x \,\}_{\substack{\partition_2(v) = g \\ f^*\partition_3(v) \neq g}}
\end{align}
Computing \eqref{g: G}, we have
\begin{align*}
	G^{\setminus \partition_2} &\overset{\eqref{a: transf G}}= G^{\setminus f^*\partition_3} \{ \,\,|_x \,\}_{\substack{\partition_2(v') = g \\ f^*\partition_3(v') \neq g}} \\
	&\overset{\eqref{a: gsd part3}}= \left( J^{\setminus \partition_3} \{ K_v, M_v \}_{\partition_3(v) \in \{b,w\}} \right) \{ \,\,|_x \,\}_{\substack{\partition_2(v') = g \\ f^*\partition_3(v') \neq g}} \\
	&= J^{\setminus \partition_3} \{ K_v \}_{\partition_2'(v) \in \{ b,w \}}\{ M_v \}_{\substack{\partition_3(v) \in \{ b,w \}\\\partition_2'(v) =g}}\{ \,\,|_x \,\}_{\substack{\partition_2(v') = g \\ f^*\partition_3(v') \neq g}} \\
	&= J^{\setminus \partition_3} \{ K_v \}_{\partition_2'(v) \in \{ b,w \}}\{ M_v \{ \,\,|_x \,\}_{\vertex(M_v)} \}_{\substack{\partition_3(v) \in \{ b,w \}\\\partition_2'(v) =g}} \\
	&= J^{\setminus \partition_3} \{ K_v \}_{\partition_2'(v) \in \{ b,w \}}\{ \,\,|_x \,\}_{\substack{\partition_3(v) \in \{ b,w \}\\\partition_2'(v) =g}}
\end{align*}
while computing \eqref{g: J} yields
\begin{align*}
	G^{\setminus \partition_2} &\overset{\eqref{a: gsd part2}}= 
	J^{\setminus \partition_2'} \{ K_v \}_{\partition_2'(v) \in \{b,w\}}  \\
	&\overset{\eqref{a: transf J}}= 
	J^{\setminus \partition_3} \{ \,\,|_x \,\}_{\substack{\partition_2'(v) = g \\ \partition_3(v) \neq g}} \{ K_v \}_{\partition_2'(v) \in \{b,w\}}
\end{align*}
so the two composites of gsd maps (\ref{g: G},\ref{g: J}) are equal.
\end{proof}

\subsection{Simplification of admissible graphs}
We now define the category of simplified graphs; a related category appears in section \ref{section filtration layers} where we order a subset of vertices.

\begin{definition} Fix an input-output profile $\dch$.
\begin{itemize} 
	\item Define $\varrho: \Ob \blood\dch \to \Ob \blood\dch$
	to be the function that turns all grey vertices to white vertices; notice that $\varrho$ takes admissible objects to admissible objects.
	\item An admissible object $(G,\partition)$ will be called \emph{simplified} if the only white, connected, non-edge subgraphs of $\varrho(G,\partition)$ are corollas (which, by admissibility, will be (1,1)-corollas). 
	\item For a fixed input-output profile $\dch$, let $\bloodsimple\dch$ be the 
	full subcategory of $\blood\dch$ 
	consisting of those $(G,\partition)$ which are both admissible and simplified.
\end{itemize}
\end{definition}

Suppose that an admissible object $(G,\partition) \in \bloodadmissible\dc$ is not simplified and is grey-free. 
Consider the set of subgraphs $L \leq G$ with
\begin{itemize}
	\item $L$ is connected and linear, 
	\item $\partition(L) = w$, and
	\item $|\vertex(L)| \geq 1$
\end{itemize}
and let $L_1, \dots, L_r$ be the set of subgraphs which are \emph{maximal} with respect to these properties and with $\coprod \vertex(L_i) = \partition^{-1}(w)$.
By maximality the $L_i$ do not intersect one another, 
and we can contract each $L_i$ to a single vertex, yielding a new graph $\overline G$.
Specifically, we have
$G = \overline G\{L_i\}_{v_i} \to \overline G$, where $v_1, \dots, v_r \subset \vertex(\overline G)$ is a set of $r$ vertices.
Define a partition on $\overline G$
\[
	\overline \partition(v) = \begin{cases}
		\partition(v) & v\in \vertex(\overline G) \setminus \{ v_1, \dots, v_r \} \subset \vertex(G)\\
		w & v \in \{ v_1, \dots, v_r \}.
	\end{cases}
\]
The gsd map $(G,\partition) \to (\overline G, \overline \partition)$ is called the \emph{simplification} of the grey-free, admissible object $(G,\partition)$.

\begin{notation}
	We will decorate our symbols for categories with $\notg$ when we only consider objects $(G,\partition)$ with $\partition(v) \in \{b,w\}$ for all $v$.
	For example, $\bloodsimple^{\notg}\dch$ is the full subcategory of $\bloodsimple\dch$ consisting of grey-free simplified graphs.
\end{notation}

\begin{lemma}\label{simplification lemma}
	If $(G,\partition) \in \bloodadmissible^{\notg}\dc$ is admissible and grey-free, then its simplification is a weakly initial object in the comma category $(G,\partition) \downarrow \bloodsimple\dc$.
\end{lemma}

\begin{proof}
	Fix a map
	\[
		(G,\partition) \overset{q}\to (G,\partition_0) \overset{f}\to (J,\partition_1)
	\]
	with $q$ a transfusion, $f$ a gsd map, and $(J,\partition_1)$ simplified.
	Since $(G,\partition)$ is grey-free, $q$ is the identity.
	Write $G = J\{K_v\}$ and $G=\overline{G} \{L_i \}_{v_i}$.
	Suppose that $v\in \vertex(J)$ with $\partition_1(v) = w$.
	Since $(J,\partition_1)$ is simplified and admissible, we know that 
	\[ |\inp(K_v)| =|\inp(v)| = 1 =  | \out(v) | = | \out(K_v) |,\] 
	and since $(G,\partition)$ is admissible we know that all white vertices $v'$ have $|\inp(v') | = |\out (v')| =1$, so $K_v$ is connected. 
	Then for each vertex $v'$ of $K_v$, $v'$ is contained in some $L_i$, but since $(J,\partition_1)$ has no adjacent white vertices, all vertices of $L_i$ are contained in $K_v$; by maximality of $L_i$ we have $L_i = K_v$.
	Notice that $K_v$ may just be an edge; write
	\[
		S = \{ \bar v_1, \dots, \bar v_r  \} \subset \vertex(J)
	\]
	for the set of vertices so that $K_{\bar v_i} \cong L_i$ has at least one vertex.
	Then we have
	\[
		\overline G \{ L_i \} = G = J\{ K_v \}_{\vertex(J)} = J\{ K_v \}_{v\notin S} \{ L_i \} 
	\]
	showing that $\overline G = J\{ K_v \}_{v\notin S}$, hence we have a factorization
	\[
		(G,\partition) \to (\overline G, \overline \partition) \to (J, \partition_1)
	\]
	of $f$.
\end{proof}

\begin{figure}
	\includegraphics[scale=0.3]{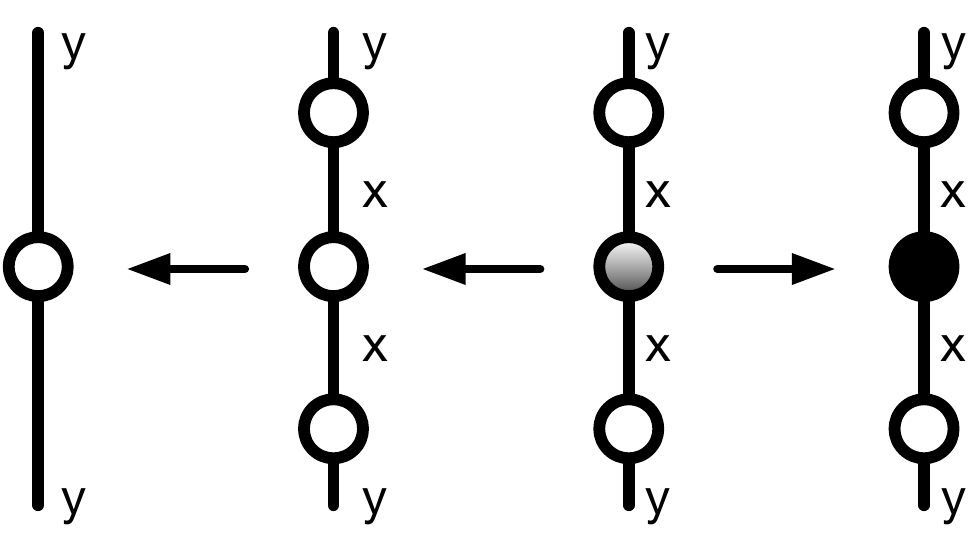}
	\caption{No maps between simplified graphs on the ends}\label{nowo}
\end{figure}

\begin{example}
	Note that if $(G,\partition)$ contains grey vertices adjacent to white vertices, then there may not be a weakly initial object in $(G,\partition) \downarrow \bloodsimple\dc$. For instance, there are no maps in $\blood\yysingle$ between the simplified graphs at the left and right ends of Figure~\ref{nowo}.
\end{example}

\begin{lemma}\label{grey free colimits}
If $D$ is the functor associated with $ \msh \leftarrow  \msi \rightarrow \msp$, 
then
	the inclusions $\bloodsimple^{\notg}\dch \to \bloodsimple\dch$ and $\bloodadmissible^{\notg}\dch \to \bloodadmissible\dch$ induce isomorphisms
	\begin{align*}
		\colim_{\bloodsimple^{\notg}\dch} D &\overset\cong\to \colim_{\bloodsimple\dch} D \\
		\colim_{\bloodadmissible^{\notg}\dch} D &\overset\cong\to \colim_{\bloodadmissible\dch} D.
	\end{align*}
\end{lemma}
\begin{proof}
We construct an inverse to 
\begin{equation}\label{internal equation grey free colimits}
\colim_{\bloodadmissible^{\notg}\dch} D \to \colim_{\bloodadmissible\dch} D,
\end{equation} 
the simplified case is similar. 
If $(G,\partition) \in \bloodadmissible\dch$, then by Lemma \ref{grey reduction no change} there is a map 
\[
	\varepsilon_{G,\partition} : D(G,\partition) \overset\cong\leftarrow D\left(\eta(G,\partition)\right) \to  \colim_{\bloodadmissible^{\notg}\dc} \decAB.
\]
If $f: (G,\partition) \to (J,\partition')$ is a morphism of $\bloodadmissible\dch$, then the diagram
\[ \begin{tikzcd}[column sep=small]
D(G,\partition) \arrow{rr}{Df} && D(J,\partition') \\
D\eta(G,\partition) \uar{\cong} \arrow{rr}{D\eta f} \arrow{dr} && D\eta(J,\partition') \uar{\cong} \arrow{dl}\\
&\colim\limits_{\bloodadmissible^{\notg}\dc} \decAB
\end{tikzcd} \]
commutes by Proposition \ref{grey red functoriality}. Hence $D(f) \varepsilon_{G,\partition} = \varepsilon_{J,\partition'}$ so $\varepsilon_{-}$ induces a map
\[\colim_{\bloodadmissible\dc} D \to \colim_{\bloodadmissible^{\notg}\dc} D.
\]
It is routine to check that this is both a left and right inverse to \eqref{internal equation grey free colimits}.
\end{proof}

\begin{proposition}\label{equal_colims}
Let $D_{\dch}=D$ be the functor $\blood\dc \to \sset$ determined by the maps $ \msi \to  \msh$ and $ \msi \to \msp$. Then
	\[
		\colim_{\bloodsimple\dc} \decAB = \colim_{\blood\dc} D
	\]
\end{proposition}
\begin{proof}
By Proposition~\ref{equal_colims1}, it is enough to show that 
\[
	\colim_{\bloodsimple\dc} \decAB \to \colim_{\bloodadmissible\dc} D
\]
is an isomorphism.
In the diagram
\[ \begin{tikzcd}
\bloodsimple^{\notg}\dc \rar \dar & \bloodsimple\dc \dar \\
\bloodadmissible^{\notg}\dc \rar & \bloodadmissible\dc 
\end{tikzcd} \]
the functor \[ \bloodsimple^{\notg}\dc \to \bloodadmissible^{\notg}\dc \] is cofinal by Lemma~\ref{simplification lemma}.
Applying Lemma \ref{grey free colimits} implies that the top and bottom maps in the diagram 
\[ \begin{tikzcd}
\colim\limits_{\bloodsimple^{\notg}\dc} D \rar{\cong} \dar{\cong}[swap]{\text{cofinality}} & \colim\limits_{\bloodsimple\dc} D \dar \\
\colim\limits_{\bloodadmissible^{\notg}\dc} D \rar{\cong} & \colim\limits_{\bloodadmissible\dc} D 
\end{tikzcd} \]
are isomorphisms, so the result follows.
\end{proof}

\section{A local filtration of the pushout}\label{section local filtration}
Consider the functions
	\begin{align*}
		\#_{wx}: \Ob \bloodsimple\dch &\to \mathbb{N} \\
		\#_{gwx} : \Ob \bloodsimple\dch &\to \mathbb{N}
	\end{align*}
with 
\begin{align*}
	\#_{wx} (G,\partition) &= \left| \left\{ v\in \vertex(G) \, 
	\mid 
	\, \partition(v) = w, \inp(v) = \out(v) = x \right\}\right| \\
	\#_{gwx} (G,\partition) &= \left| \left\{ v\in \vertex(G) \, 
	\mid 
	\, \partition(v) \in \{g,w \}, \inp(v) = \out(v) = x\right\}\right|.
\end{align*}
Let
$\bloodsimple_{\leq k}\dc$ be the full subcategory of $\bloodsimple\dc$ with objects $\#_{wx}^{-1}[0,k]$.
Note that all maps of $\bloodsimple\dc$ are non-decreasing with respect to $\#_{wx}$, so $\bloodsimple_{\leq k}\dc$ is a \emph{right ideal} of $\bloodsimple\dc$, meaning it is 
 is closed under precomposition with maps from $\bloodsimple\dc$.
Evidently $\bloodsimple\dc = \colim_k \bloodsimple_{\leq k}\dc.$

\begin{definition}\label{witnessing definition}
Suppose that $G : \calC \to \calD$ and $F: \calD \to \sSet$ are functors.
If $c\in \calC$ and $z \in FG(c)$, write $[z]_\calC \in \colim_\calC F$ for its image. 
Likewise, if $z\in F(d)$ write $[z]_\calD \in \colim_\calD F$.
Write 
\begin{align*}
	\colim_{\calC} FG &\overset{g}\to \colim_{\calD} F \\
	[z]_\calC &\mapsto [z]_\calD
\end{align*}
for the induced map on colimits.
To study injectivity of $g$, it is helpful to introduce a sequence of relations $\sim_n$.

Suppose that $g(x) = g(x')$. 
A \emph{witness} to $x\sim_n x'$ consists of 
a zigzag of morphisms
\begin{equation} \label{zigzag}
	G(c) \overset{f_1}\rightarrow d_1 \overset{f_2}\leftarrow d_2 \overset{f_3}\rightarrow \dots \overset{f_{2n-1}}\rightarrow d_{2n-1} \overset{f_{2n}}\leftarrow G(c')
\end{equation}
in $\calD$, 
representatives $z\in FG(c)$, $z'\in FG(c')$ with $[z]_\calC = x, [z']_\calC = x'$ (where $c,c' \in \calC$),
and elements 
$z_i \in F(d_i)$ so that
\begin{align*}
	F(f_1)(z) &= z_1 & F(f_{2n})(z') &= z_{2n-1} \\
	F(f_{2i})(z_{2i}) &= z_{2i-1} & F(f_{2i+1})(z_{2i}) &= z_{2i+1}.
\end{align*}
There is a corresponding function
\begin{align*}
\omega &:  \left(\colim_{\calC} FG\right) \times_g \left(\colim_{\calC} FG\right) \to  \mathbb{N} \\
		\omega(x,x') &= \min \{ n \mid x \sim_n x' \}.
\end{align*}

\end{definition}

If $x\sim_n x'$ then $x\sim_{n+1} x'$, which is witnessed by extending \eqref{zigzag} with $f_{2n+1} = f_{2n+2} = \id_{Gc'}$ and $z_{2n} = z_{2n+1} = z'$. 

\begin{remark}\label{omega identically zero}
	Suppose that $\calC$ is a full subcategory of $\calD$ 
	
	and $F: \calD \to \sSet$ a functor.
	Consider the relation $x\sim_n x'$ on $\colim_{\calC} F$. 
	The case $n=0$ corresponds to equality in $\colim_\calC F$, so $x=x'$ if and only if $x\sim_0 x'$.
	The map $g: \colim_{\calC} F \to \colim_{\calD} F$ is injective if and only if $\omega(-,-)$ is identically zero.
\end{remark}

\begin{notation}
	Let $\incident(G)$ be the set of vertices of $G$ which are incident to $y$-edges. Notice that if $(G,\partition)$ is simplified, then all such vertices are extremal.
	
\end{notation}

\begin{proposition}
If $D$ is the functor associated with $ \msh \leftarrow  \msi \rightarrow \msp$, then
	\[
		\colim_{\bloodsimple_{\leq \ell}\dc} D \to \colim_{\bloodsimple\dc} D
	\]
	is a monomorphism.
\end{proposition}
\begin{proof}
		
    Suppose that $x,x' \in \colim_{\bloodsimple_{\leq \ell}} D$ with $x$ and $x'$ mapping to the same element in $\colim_{\bloodsimple} D$ and
	\[
				x\neq x' \qquad\overset{\ref{omega identically zero}}{\Leftrightarrow}\qquad \omega(x, x') = n > 0.
	\]

	Choose a zigzag \eqref{zigzag} that witnesses $x\sim_n x'$, which begins as 
	\begin{equation}\label{ourzigzag}
		(H,\partition') \overset{f}\rightarrow (G,\partition) \overset{f'}\leftarrow (H',\partition'') \overset{f_3}\rightarrow d_3 \leftarrow \dots \overset{f_{2n}}\leftarrow c'
	\end{equation}
	with  $\#_{wx}(H,\partition') = k \leq \ell$. There is $z\in D(H,\partition')$, $z'\in D(H',\partition'')$ so that $[z]_\calC = x$ and $D(f)(z) = D(f')(z')$.
	By Lemma \ref{grey reduction no change} and the fact that $\bloodsimple_{\leq k} \dc$ is a right ideal, we may assume that $(H,\partition')$ is grey-free.

The morphism
	\[
		f: (H,\partition') \to (G,\partition)
	\]
	is in $\bloodsimple\dc$. 
	Since $(H,\partition')$ is grey-free, we have that $f$ is a gsd map. Further, since $(H,\partition')$ is simplified, we have
	\[
		H = G\{ K_v \}_{\partition^{-1}(b) \cup \incident(G)} \{ \,\, |_{x} \, \}_{S}
	\]
	for some set $S$ of white and grey vertices of $(G,\partition)$ with
	\[ \partition^{-1}(g) \subset S \subset \partition^{-1}(g) \cup \partition^{-1}(w) \setminus \incident(G)\] 
	and $\profilev = \xxsingle$ for every $v\in S$.
	Let $\partition_S$ be the partition of $G$ defined by
	\[
		\partition_S(v) = \begin{cases}
			\partition(v) & v \notin S \\
			g & v \in S;
		\end{cases}
	\]
	notice that we have
	\begin{align*}
		\partition_S^{-1}(w) &= \partition^{-1}(w) \setminus S \\
		\partition_S^{-1}(g) &= \partition^{-1}(g) \cup S  = S\\
		\partition_S^{-1}(b) &= \partition^{-1}(b).
	\end{align*}
	We have $
		G^{\setminus \partition_S} = G\{ \,\,|_x \,\}_{\partition_S(v) = g} = G\{ \,\,|_x \,\}_{S}$
	hence $
		H = G^{\setminus \partition_S} \{K_v \}_{\partition_S^{-1}(b) \cup \incident(G)}.$
	Then $f$ factors as
	\[
		(H,\partition') \to \eta(G,\partition_S) \to (G,\partition_S) \to (G,\partition);
	\]
	note that $\#_{wx} (H,\partition') =\#_{wx} \eta(G,\partition_S) = k$.
	The image of $D(f)$ is equal to $D\eta(G,\partition_S) = D(G,\partition_S) \subset D(G,\partition)$.

	We may also assume that $(H', \partition'')$ is grey-free and 
	$\#_{wx}(H', \partition'') = j \leq \ell$ (otherwise precompose with $gr(H',\partition'')$).
Factor the map
	\[ f' : (H', \partition'') \to (G,\partition) \]
	 as before for some $\partition^{-1}(g) \subset T\subset \partition^{-1}(g) \cup \partition^{-1}(w)$
	\[
		(H',\partition'') \to \eta(G,\partition_T) \to (G,\partition_T) \to (G,\partition)
	\]
	where the last map is grey-white transfusion on $T\cap \partition^{-1}(w)$, the middle map is a gsd map that creates grey vertices, and the first map preserves $\#_{wx}$. The image of $D(f')$ is $D\eta(G,\partition_T) = D(G,\partition_T) \subset D(G,\partition)$.
	Notice that $D(G,\partition_S) \cap D(G,\partition_T) = D(G,\partition_{S\cup T})$. We have a diagram
	\[ \begin{tikzcd}
	\eta (G,\partition_S) \rar & (G,\partition_S) \rar & (G,\partition) & (G,\partition_T) \lar & \eta(G,\partition_T) \lar \\
	&& (G,\partition_{S\cup T}) \arrow{ur}[swap]{tr} \arrow{ul}{tr} \uar \\
	&&\eta(G,\partition_{S\cup T}). \uar \arrow{uurr}[swap]{gsd} \arrow{uull}{gsd}
	\end{tikzcd} \]
	The outer diagonal maps either create white vertices or are identities.

Notice that $\#_{wx} \eta(G,\partition_{S\cup T}) \leq \min (k,j) \leq k \leq \ell$.
Write
\begin{align*}
	\eta(G,\partition_T) &= (G^{\setminus \partition_T},\partition_T^{\setminus g}) \\
	\eta(G,\partition_{S\cup T}) &= (G^{\setminus \partition_{S\cup T}}, \partition_{S\cup T}^{\setminus g}) & G^{\setminus \partition_{S\cup T}} &= G^{\setminus \partition_T}\{\,\,|_x\,\}_U \\
	&& U &= (S\cap \partition^{-1}(w)) \setminus T\\ &&&\subset (\partition_T^{\setminus g})^{-1}(w) = \partition_T^{-1}(w).
\end{align*}
Now $(H',\partition'') \to \eta(G,\partition_T)$ is a gsd map, and we write
\[
	H' = G^{\setminus \partition_T}\{K_v\}_{\left(\partition_T^{\setminus g}\right)^{-1}(b) \cup \incident\left(G^{\setminus \partition_T}\right)},
\]
where $\incident\left(G^{\setminus \partition_T}\right) \cong \incident(G)$.
Note that $U\cap \left(\left(\partition_T^{\setminus g}\right)^{-1}(b) \cup \incident\left(G^{\setminus \partition_T}\right)\right) = \varnothing$, so we can form a new graph $H''$ with $\vertex(H'') = \vertex(H') \setminus U$ by setting
\begin{align*}
	H'' &= G^{\setminus \partition_T}\{K_v\}_{\left(\partition_T^{\setminus g}\right)^{-1}(b) \cup \incident\left(G^{\setminus \partition_T}\right)}\{\,\,|_x\,\}_U \\
	&= H'\{\,\,|_x\,\}_U \\
	&= G^{\setminus \partition_{S\cup T}}\{K_v\}_{(\partition_T^{\setminus g})^{-1}(b) \cup \incident\left(G^{\setminus \partition_T}\right)}.
\end{align*}
We have the diagram of gsd maps
\[ \begin{tikzcd}
\eta(G,\partition_{ T}) & (H', \partition'')\lar
\\
\eta(G,\partition_{S\cup T})\uar & (H'', \partition''|_{H''})\uar{g}\lar
\end{tikzcd} \]
and the objects on the bottom in $\bloodsimple^{\notg}_{\leq k}\dch$.
Applying $D$, we have
\begin{center}\includegraphics{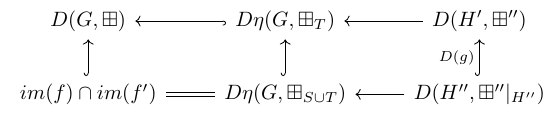}\end{center}
with the right square a pullback. 

We've replaced the chain
\[
	(H,\partition') \rightarrow (G,\partition) \leftarrow (H',\partition'')
\]
with the longer chain
\[
	(H, \partition') \rightarrow \eta(G,\partition_S) \leftarrow (H'',\partition''|_{H''}) \rightarrow (H',\partition'')
\]
with the three objects on the left in $\bloodsimple_{\leq k}$.
However, $z'\in D(H'', \partition''|_{H''})$, so we can modify the  zigzag \eqref{ourzigzag} by looking at
\begin{equation}
		(H'',\partition''|_{H''}) \overset{f_3 \circ g}\rightarrow d_3 \leftarrow \dots \overset{f_{2n}}\leftarrow c'
\end{equation}
which witnesses that $x=[z]_{\bloodsimple_{\leq \ell}} = [z']_{\bloodsimple_{\leq \ell}} \sim_{n-1} x'$, contrary to the assumption that $\omega(x,x') = n$.

It follows that $\omega(x,x') = 0$ for all pairs $x,x'$, so by remark \ref{omega identically zero} the map is injective.
\end{proof}

\begin{corollary}
If $D$ is the functor associated with $ \msh \leftarrow  \msi \rightarrow \msp$, then
	\[
		\colim_{\bloodsimple_{\leq k}\dc} D \to \colim_{\bloodsimple_{\leq k+1}\dc} D
	\]
	is a monomorphism. \qed
\end{corollary}

\begin{notation}
	Write
	\[
		\msq^k \dc = \colim_{\bloodsimple_{\leq k}\dc} D;
	\]
	then for each $\dc$ we have an exhaustive filtration
	\[
		\msq^0 \dc \hookrightarrow \msq^1 \dc \hookrightarrow \cdots \hookrightarrow \msq^k \dc \hookrightarrow \cdots \hookrightarrow \colim_n \msq^n \dc.
	\]
\end{notation}

\section{Exhibition of the filtration layers as a pushout}
\label{section filtration layers}

We will now consider graphs which come equipped with an ordering of a subset of $k$ vertices; by considering maps which preserve these vertices and their ordering we end up with a category $\marked_k\dc$.
An object of $\marked_k\dc$ is a pair $(G,\ell)$ with $G\in \sgc\dc$ and $\ell:  \{ 1,\dots, k \} \to \vertex(G)$ satisfying
\begin{itemize}
	\item all $y$-edges of $G$ are extremal
	\item the set of vertices incident to $y$-edges, denoted by $\incident(G)$, has the property that if $v\in \incident(G)$ then $\profilev \in \left\{ \yxsingle, \xysingle, \yysingle \right\}$,
	\item all vertices $v\in \image \ell$ have $\profilev = \xxsingle$
	\item no two vertices in $\image \ell \cup \incident(G)$ are adjacent.
\end{itemize}
For the morphism spaces, we take
\[
	\hom_{\marked_k\dc} ((G,\ell), (H,\ell')) \subset \hom_{\sgc\dc} (G,H)
\]
to be those decompositions $f$ of the form
\begin{align*}
	G &= H\{K_v \}_{v\notin \image \ell'} \\
	f_{\vertex} \circ \ell & = \ell'.
\end{align*}
There is a functor
\begin{align*}
	A : [1]^k \times \marked_k\dc &\to \bloodsimple_{\leq k} \dc \\
	(t_1, \dots, t_k,(G,\ell)) & \mapsto (G,\partition)
\end{align*}
where $\partition$ is defined by
\[
	\partition(v) = \begin{cases}
		g & v = \ell(i) \text{ and } t_i = 0 \\
		w & v = \ell(i) \text{ and } t_i = 1 \\
		w & v \in \incident(G) \\
		b & \text{otherwise.}
	\end{cases}
\]
The (multi)functor $A$ takes $(constant, 0 \to 1, constant)$ to an atomic white transfusion and $(constant, (G,\ell) \to (H,\ell'))$ to the gsd map given by $G\to H$.
Let $[1]^k_*$ denote the full subcategory of $[1]^k$ without the terminal object $1^k$.
Notice that the restriction of $A$ to $[1]^k_* \times \marked_k\dch$ lands in $\bloodsimple_{\leq k-1} \dch$ and that every object of $\#_{gwx}^{-1}(k)$ is in the image of $A$.

Define
\begin{align*}
	\precrush^\Ononempty_k\dc &= \colim_{[1]^k_* \times \marked_k\dc} DA \\
	\precrush_k\dc &= \colim_{[1]^k \times \marked_k\dc} DA. \\
\end{align*}
We can write the composite
\[
	[1]^k \times \marked_k\dc \overset{A}\to \bloodsimple_{\leq k} \dc \overset{D}{\to} \sSet
\]
as a product $DA \cong D_L \times D_R$, where
\begin{align*}
	D_L : [1]^k &\to \sSet & Z_0 &= \msi\xxsingle = *\\
	(t_1, \dots, t_k) &\mapsto \prod_{i=1}^k Z_{t_i} & Z_1 &= \msh\xxsingle \simeq *
\end{align*}
and
\begin{align*}
	D_R &: \marked_k\dc \to \sSet \\
	D_R(G,\ell) &= \left( \prod_{v \notin \incident(G) \cup \image \ell} \msp \profilev \right) \times \left( \prod_{v \in \incident(G)} \msh \profilev \right).
\end{align*}
Thus we have
\[
	\precrush^\Ononempty_k\dc = \colim_{[1]^k_* \times \marked_k\dc} DA = \colim_{[1]^k_* \times \marked_k\dc} (D_L \times D_R) = \colim_{[1]^k_*} D_L \times \colim_{\marked_k\dc} D_R
\]
and 
\[ \precrush_k\dc  = \colim_{[1]^k} D_L \times \colim_{\marked_k\dc} D_R. \]
Note that there is a free $\Sigma_k$-action on $[1]^k \times G_k^M\dch$,  which induces a free $\Sigma_k$-actions on
$\precrush_k\dch$ and $\precrush^\Ononempty_k\dch$.

\begin{lemma}[Pushout-Product Lemma]\label{iterated pushout product}
  Suppose that $F: [1] \to \sSet$ takes the generating morphism $0\to 1$ to a(n acyclic) cofibration.
  
  Then, for $n\geq 2$, the map
  \[
  	\colim_{[1]^n_*} F^{\times n} \to \colim_{[1]^n} F^{\times n} \cong F(1)^{\times n}
  \]
  is a(n acyclic) cofibration as well.
\end{lemma}
\begin{proof}
	A simple induction using the usual pushout-product axiom.
\end{proof}

\begin{lemma}\label{using pushout product lemma}
	The map
	\[
		\precrush^\Ononempty_k\dc \to \precrush_k\dc
	\]
	is an acyclic cofibration in $\sSet$.
	Furthermore, it induces an acyclic cofibration
	\[
		\precrush^\Ononempty_k\dc / \Sigma_k \to \precrush_k\dc / \Sigma_k.
	\]
\end{lemma}
\begin{proof}
	By assumption, $\msi\xxsingle \to \msh\xxsingle$ is an acyclic cofibration. The pushout-product lemma \ref{iterated pushout product} then gives that
	\[ \colim_{[1]^k_*} D_L \to \colim_{[1]^k} D_L \]
	is an acyclic cofibration as well.

	For the second statement, note that the map is $\Sigma_k$-equivariant and that the $\Sigma_k$-actions on $\precrush^\Ononempty_k\dch$ and $\precrush_k\dch$ are free.
\end{proof}

The diagram
\[ \begin{tikzcd}
		\precrush^\Ononempty_{k+1}\dc \dar \rar & \msq^k\dc \dar \\
		\precrush_{k+1}\dc \rar & \msq^{k+1}\dc.
\end{tikzcd} \]
factors as
\[ \begin{tikzcd}
		\precrush^\Ononempty_{k+1}\dc \dar \rar &  \precrush^\Ononempty_{k+1}\dc / \Sigma_{k+1} \dar \rar & \msq^k\dc \dar \\
		\precrush_{k+1}\dc \rar & \precrush_{k+1}\dc / \Sigma_{k+1} \rar & \msq^{k+1}\dc
\end{tikzcd} \]
by forgetting the ordering of the marked vertices.

\begin{proposition}\label{big pushout proposition}
The diagram
	\[ \begin{tikzcd}
		\precrush^\Ononempty_{k+1}\dc / \Sigma_{k+1} \dar \rar & \msq^k\dc \dar \\
		\precrush_{k+1}\dc / \Sigma_{k+1} \rar & \msq^{k+1}\dc.
\end{tikzcd} \]
is a pushout.
\end{proposition}
\begin{proof}
	Consider the map
	\begin{equation}\label{map to show is iso}
		C: \precrush_{k+1}\dc \amalg_{\precrush^\Ononempty_{k+1}\dc} \msq^k\dc \to \msq^{k+1}\dc.
	\end{equation}
	For each element of $\msq^{k+1}\dc$ we can find a representative $z \in D(G,\partition)$ for some $(G,\partition) \in \bloodsimple^{\notg}_{\leq k+1}$. If $\#_{wx}(G,\partition) \leq k$, then $[z]$ is in the image of $\msq^k\dc$, while if $\#_{wx}(G,\partition) = k+1$ then $(G,\partition) = A(1,\dots, 1, (G,\ell))$ for a choice of ordering $\ell$ of the white vertices of $(G,\partition)$, so $[z]$ is in the image of $\precrush_{k+1}\dc$. Thus \eqref{map to show is iso} is surjective.

	It is not true that \eqref{map to show is iso} is injective,
	but we will show that 
   \begin{equation*}
		C_{\Sigma_{k+1}}: \precrush_{k+1}\dc / \Sigma_{k+1} \uamalg{\precrush^\Ononempty_{k+1}\dc / \Sigma_{k+1}} \msq^k\dc \to \msq^{k+1}\dc.
	\end{equation*}
	is injective.
	Suppose that we have
	\[
		x,x' \in \precrush_{k+1}\dc \setminus \left( \image \precrush^\Ononempty_{k+1}\dc \right)
	\]
	with $C(x) = C(x')$. If we can show that this implies $x=x' \mod \Sigma_k$, then $C_{\Sigma_{k+1}}$ is injective.

Suppose that $C(x) = C(x')$ and, using the notation of definition \ref{witnessing definition}, that $x\sim_n x'$ for some $n > 0$.
Choose a witness to $x\sim_n x'$, i.e.\ elements
	\begin{align*}
		z & \in DA(c) & z' & \in DA(c') \\
		[z] &= x \in \precrush_{k+1}\dc & [z'] &= x' \in \precrush_{k+1}\dc,
	\end{align*} 
	a zigzag
\begin{equation}\label{new zigzag}
	A(c) \overset{f_1}\rightarrow d_1 \overset{f_2}\leftarrow d_2 \overset{f_3}\rightarrow \dots \overset{f_{2n-1}}\rightarrow d_{2n-1} \overset{f_{2n}}\leftarrow A(c')
\end{equation}
in $\bloodsimple_{\leq k+1}\dch$, and 
elements 
$z_i \in D(d_i)$ so that
\begin{align*}
	D(f_1)(z) &= z_1 & D(f_{2n})(z') &= z_{2n-1} \\
	D(f_{2i})(z_{2i}) &= z_{2i-1} & D(f_{2i+1})(z_{2i}) &= z_{2i+1}.
\end{align*}
Note that remark \ref{omega identically zero} does not apply in the present setting; nevertheless, we will show that $\omega(x,x') \leq 1$.

Our goal is to show that if $n>1$, $x\sim_n x'$ implies $x\sim_{n-1} x'$. 
Write the beginning of \eqref{new zigzag} as
\[
	A(1,\dots, 1, (H,\ell')) = (H, \partition') \overset{f}\to (G,\partition) \overset{f'}\leftarrow (H', \partition'');
\]
note that $(H,\partition')$ is grey-free, otherwise our $x$ is in the image of $\precrush^\Ononempty_{k+1}\dc$.
Hence $\#_{wx}(H,\partition') = k+1$ and $f$ is a gsd map.
Notice that $f$ cannot increase $\#_{wx}$ (otherwise we leave $\bloodsimple_{\leq k+1}$), so $f$ is given by
\[
	H = G\{K_v\}_{\partition^{-1}(b) \cup \incident(G)}\{\,\,|_x\,\}_{\partition(v) = g}.
\]
We then have a diagram
\[ \begin{tikzcd}
& \eta(G,\partition) \dar & \eta(H', \partition'') \dar \lar[swap]{\eta(f')} \\
(H,\partition') \rar{f} \arrow[dashed]{ur}{\tilde f}& (G,\partition) & (H', \partition'')  \lar{f'}
\end{tikzcd} \]
where the gsd map $\tilde f$ is given by
$H = G^{\setminus \partition} \{K_v \}_{\partition^{-1}(b) \cup \incident(G)}$.
Notice that $\tilde f$ is the image under $A$ of a map
\[ \id \times \tilde f : (1,\dots, 1, (H,\ell')) \to (1, \dots, 1, (G^{\setminus \partition}, \ell)), \]
where $\ell = \tilde f_{\vertex} \circ \ell'$.

Factor $f'$ as a transfusion $q$ followed by a gsd map $h$
\begin{equation}\label{internal eq factorization}
	f' : (H',\partition'') \overset{q}\to (H', \partition_0) \overset{h}\to (G,\partition).
\end{equation}
Write $h$ to be the decomposition
\[
	H' = G\{ L_v \}_{\partition(v) \in \{ b,w \}} \{ R_v \}_{\partition(v) = g};
\]
then by (\ref{e: G to G0},\ref{e: gr G0 to H}) and the fact that everything is simplified, $\eta(h)$ is the 
decomposition
\begin{equation}\label{e: eta g decomp}
	H'^{\setminus \partition_0} = G^{\setminus \partition}\{ L_v \}_{\partition(v) \in \{ b,w \}}.
\end{equation}
Let $S$ be the set of vertices $v$ with 
$\partition({v}) = w$, ${v}\notin \incident(G)$, and $L_{v} = |_x$.
We will argue that $\eta(h)$ is the image under $A$ of a morphism of the form
\begin{equation}\label{e: preimage eta g}
	(1,\dots, 1, (H'^{\setminus \partition_0},\ell_0)) \to (1, \dots, 1, (G^{\setminus \partition}, \ell)).
\end{equation}
Define a new partition $\partition_S$ on $G^{\setminus \partition}$ by 
\[
	\partition_S(v) = \begin{cases}
		g & v \in S; \\
		\partition^{\setminus g}(v) & \text{otherwise.}
	\end{cases}
\]
But now
\begin{align*}
	H'^{\setminus \partition_0} &= G^{\setminus \partition}\{ L_v \}_{\partition(v) \in \{ b,w \}} \\ 
	&= (G^{\setminus \partition})^{\setminus \partition_S} \{ L_v \}_{\partition^{-1}\{ b,w \} \setminus S}.
\end{align*}
giving the first map in the factorization
\begin{equation}\label{decomposition of eta g}
	\eta(H', \partition_0) \to \eta(G^{\setminus \partition}, \partition_S) \to (G^{\setminus \partition}, \partition_S) \to (G^{\setminus \partition}, \partition^{\setminus g}) = \eta(G,\partition).
\end{equation}
The map
\[
	(G^{\setminus \partition}, \partition_S) \to (G^{\setminus \partition}, \partition^{\setminus g})
\]
is the image under $A$ of a map
\begin{equation}\label{map right now}
	(t_1, \dots, t_{k+1}, (G^{\setminus \partition}, \ell)) \to (1,\dots, 1, (G^{\setminus \partition}, \ell)) 
	\end{equation}
	\[
	t_i = \begin{cases}
		0 & \ell(i) \in S \\
		1 & \text{otherwise.}
	\end{cases}
\]
If this is \emph{not} the identity then $x$ is represented by an element of $DA(t_1, \dots, t_{k+1}, (G^{\setminus \partition}, \ell))$, hence is in the image of $\precrush_{k+1}^+\dch$, contrary to assumption. Thus \eqref{map right now} is the identity,  so $t_i = 1$ for all $i$ and $S=\varnothing$.
Now, in \eqref{decomposition of eta g}, we have
\[ \begin{tikzcd}
	\eta(H', \partition_0) \rar & 
	\eta(G^{\setminus \partition}, \partition_S) \rar \arrow{dr}{=} & 
	(G^{\setminus \partition}, \partition_S) \rar{=} & 
	(G^{\setminus \partition}, \partition^{\setminus g}) \dar{=}\\
	&
	&
	\eta(G^{\setminus \partition}, \partition_S) \uar \rar{=}& 
	\eta(G,\partition)
\end{tikzcd} \]
	which shows that the map $\eta(h): \eta(H', \partition_0)  \to \eta(G,\partition)$ from \eqref{e: eta g decomp}
	is of the form
	\[
		H'^{\setminus \partition_0} = G^{\setminus \partition}\{ L_v \}_{\partition^{-1}(b) \cup \incident(G)},
	\] hence we have that $\eta(h)$ is the image under $A$ of a morphism of $\id \times \eta(h)$ as in \eqref{e: preimage eta g}.

	Next consider $q: (H',\partition'') \to (H', \partition_0)$ from \eqref{internal eq factorization}; by \eqref{e:gr transf} we have that $\eta(q)$ is the decomposition
\[ 
	H'^{\setminus \partition''} = H'^{\setminus \partition_0} \{ \,\,|_x \,\}_{\substack{\partition''(v) = g \\ \partition_0(v) \neq g}}.
\]
	So far we know that $\eta(H',\partition_0) = A(1,\dots, 1, (H'^{\setminus \partition_0},\ell_0))$
	We play the same kind of game as above and let 
\[
	S = \left\{ v \in \vertex(H'^{\setminus \partition_0} ) \,\,\middle|\,\, \partition''(v) = g  \,\, \& \,\, \partition_0(v) \neq g \right\}
\]
so that 
$H'^{\setminus \partition''} = H'^{\setminus \partition_0} \{ \,\,|_x \,\}_{S}$.
	Define a partition on $H'^{\setminus \partition_0}$ by 
	\[
		\partition_S(v) = \begin{cases}
			g & v \in S \\
			\partition_0^{\setminus g}(v) & \text{otherwise,}
		\end{cases}
	\]
so we have a factorization
\begin{equation}\label{e: decomp of eta q} \eta(H',\partition'') \to (H'^{\setminus \partition_0},\partition_S) \to \eta(H',\partition_0)
\end{equation}
where the second map is a transfusion and the first map creates grey vertices indexed by $S$.
The transfusion is the image under $A$ of
\[
	(t_1, \dots, t_{k+1}, (H'^{\setminus \partition_0},\ell_0)) \to (1,\dots, 1, (H'^{\setminus \partition_0},\ell_0))
\]
where $t_i = 0$ if $\ell_0(i) \in S$ and $t_i = 1$ if $\ell_0(i) \notin S$. But if any $t_i = 0$ then $x$ is in the image of $\precrush^+_{k+1} \dch$, hence $S = \varnothing$ and both maps in \eqref{e: decomp of eta q} are identities.

Let us return now to our witness of $x\sim_n x'$:
\[
	A(1,\dots,1,(H,\ell')) = (H, \partition') \overset{f}\to (G,\partition) \overset{f'}\leftarrow (H', \partition'') \overset{f_3}\rightarrow \dots \overset{f_{2n-1}}\rightarrow d_{2n-1} \overset{f_{2n}}\leftarrow A(c')
\]

\begin{center}\includegraphics{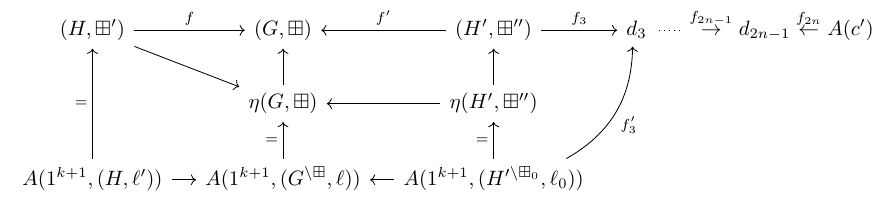}\end{center}
Since $z_2$ actually lives in $D\eta(H',\partition'')$ by Lemma \ref{grey reduction no change}, we get a witness
\[
	A(1, \dots, 1, (H'^{\setminus \partition_0},\ell_0)) \overset{f_3'}\rightarrow d_3 \overset{f_4}\leftarrow \dots \overset{f_{2n-1}}\rightarrow d_{2n-1} \overset{f_{2n}}\leftarrow A(c')
\]
to $x = [z_2] \sim_{n-1} x'$.

All of the above constructions work when $n=1$, except now the witness to $x\sim_1 x'$ looks like
\[
	A(1,\dots,1,(H,\ell')) = (H, \partition') \overset{f}\to (G,\partition) \overset{f'}\leftarrow (H', \partition'') = A(1,\dots,1,(H',\ell''))
\]
with $(H',\partition'')$ grey-free. The diagram above reduces to
\[ \begin{tikzcd}[column sep=small] 
&& A(1^{k+1},(H',\ell'')) \dar{=} \\
(H, \partition') \rar{f} & (G,\partition)&   \lar[swap]{f'}  (H', \partition'')  &  \\
A(1^{k+1},(H,\ell')) \uar{=} \rar & A(1^{k+1}, (G^{\setminus \partition}, \ell)) \uar & A(1^{k+1}, (H',\ell_0)).\uar{=} \lar 
\end{tikzcd} \]
But of course we don't have control over $\ell_0$, which may be different from $\ell''$, so the
composite $A(1,\dots,1, (H',\ell_0)) \to A(1,\dots,1,(H',\ell''))$ does not necessarily come from a map $(1,\dots,1, (H',\ell_0)) \to (1,\dots,1,(H',\ell''))$ in $[1]^k \times \marked_{k+1}\dch$. But $\image(\ell'')= \partition''^{-1}(w) = \image(\ell_0)$,
so $\tilde z_2 \in DA(1^{k+1}, (H',\ell_0))$
and $z_2 \in DA(1^{k+1}, (H',\ell''))$ map to the same element in 
$\precrush_{k+1}\dc / \Sigma_{k+1}$. Thus $x=[z]=[\tilde z_2]$ and $x'=[z_2] = \sigma [\tilde z_2]$ give the same element of $\precrush_{k+1}\dc / \Sigma_{k+1}$, so $C_{\Sigma_{k+1}}$ is injective.
\end{proof}

\begin{proof}[{Proof of Proposition \ref{NAFA}}]

	Fix an input-output profile $\dch$. Then we have isomorphisms
\[
	\colim_k \msq^k\dc = \colim_k \colim_{\bloodsimple_{\leq k}\dc} D \cong \colim_{\bloodsimple\dc} D 
	\overset{\ref{equal_colims}}\cong \colim_{\blood\dc} D \overset{\ref{pushout vs colim}}\cong \msh \amalg_{\msi} \mst \dch.
\]
	By Lemma~\ref{using pushout product lemma} and Proposition~\ref{big pushout proposition}, each
	\[
		\msq^k\dch \to \msq^{k+1}\dch
	\]
	is an acyclic cofibration of simplicial sets.
	Hence
	\[
		\msq^0 \dch \to \colim_k \msq^k \dch \cong \msh \amalg_{\msi} \mst \dch
	\]
	is an acyclic cofibration as well.
	If $\dch$ is a $\col(\mst)$ profile, then $\sG_{\col(\mst)} = \bloodsimple_{\leq 0} \dch$, so $\mst\dch = \msq^0 \dch$ by Proposition \ref{colim sgc}.
	Hence
	\[
		\mst\dch \to \msh \amalg_{\msi} \mst \dch
	\]
	is a weak equivalence for all input-output profiles $\dch$ in $\col(\mst)$, so $\mst \to \msh \amalg_{\msi} \mst$ satisfies (W1).
\end{proof}

\bibliographystyle{amsplain}
\bibliography{infinity}

\end{document}

%% file: graphs.tex

\newsavebox\greytangle
\begin{lrbox}{\greytangle}
\begin{tikzpicture} 
	\draw[thick] plot [smooth] coordinates {(0,0) (0,0.5) (0.5,1) (1, 1.5) (1,2)};
	\draw[draw=white,double=black,very thick] plot [smooth] coordinates { (0,2) (0,1.5) (0.5,1) (1,0.5) (1,0)};
	\draw[thick] plot coordinates {(1.5,0) (1.5,2)};
	\shadedraw [shading=axis] (0,1.5) circle (0.75ex);
	\shadedraw [shading=axis] (0,0.5) circle (0.75ex);
	\shadedraw [shading=axis] (1.5,1) circle (0.75ex);
	\draw[dashed] (-0.25,2) -- (1.75,2);
	\draw[dashed] (-0.25,0) -- (1.75,0);
\end{tikzpicture}
\end{lrbox}

\newsavebox\blacktangle
\begin{lrbox}{\blacktangle}
\begin{tikzpicture} 
	\draw[thick] plot [smooth] coordinates {(0,0) (0,0.5) (0.5,1) (1, 1.5) (1,2)};
	\draw[draw=white,double=black,very thick] plot [smooth] coordinates { (0,2) (0,1.5) (0.5,1) (1,0.5) (1,0)};
	\draw[thick] plot coordinates {(1.5,0) (1.5,2)};
	\draw [draw,fill] (0,1.5) circle (0.75ex);
	\draw [draw,fill] (0,0.5) circle (0.75ex);
	\draw [draw,fill] (1.5,1) circle (0.75ex);
	\draw[dashed] (-0.25,2) -- (1.75,2);
	\draw[dashed] (-0.25,0) -- (1.75,0);
\end{tikzpicture}
\end{lrbox}

\newsavebox\corollablackvariant
\begin{lrbox}{\corollablackvariant}
\begin{tikzpicture} 
	\draw[thick] (0,0) -- (0.75,1) -- (0,2);
	\draw[thick] (1.5,2) -- (0.75,1) -- (1,2);
	\draw[thick] (1.5,0) -- (0.75,1) -- (1,0);
	\draw [draw,fill] (0.75,1) circle (0.75ex);
	\draw[dashed] (-0.25,2) -- (1.75,2);
	\draw[dashed] (-0.25,0) -- (1.75,0);
\end{tikzpicture}
\end{lrbox}

\newsavebox\corollathreesblack
\begin{lrbox}{\corollathreesblack}
\begin{tikzpicture} 
	\draw[thick] (0,0) -- (0.5,1) -- (0,2);
	\draw[thick] (0.5,2) -- (0.5,1) -- (1,2);
	\draw[thick] (0.5,0) -- (0.5,1) -- (1,0);
	\draw [draw,fill] (0.5,1) circle (0.75ex);
	\draw[dashed] (-0.25,2) -- (1.25,2);
	\draw[dashed] (-0.25,0) -- (1.25,0);
\end{tikzpicture}
\end{lrbox}

\newsavebox\corollathreesgrey
\begin{lrbox}{\corollathreesgrey}
\begin{tikzpicture} 
	\draw[thick] (0,0) -- (0.5,1) -- (0,2);
	\draw[thick] (0.5,2) -- (0.5,1) -- (1,2);
	\draw[thick] (0.5,0) -- (0.5,1) -- (1,0);
	\shadedraw [shading=axis] (0.5,1) circle (0.75ex);
	\draw[dashed] (-0.25,2) -- (1.25,2);
	\draw[dashed] (-0.25,0) -- (1.25,0);
\end{tikzpicture}
\end{lrbox}

\newsavebox\threeblowups
\begin{lrbox}{\threeblowups}
\begin{tikzpicture} 
	\draw[thick] plot [smooth] coordinates {
			(0.25,4) 
			(0.25,0)
	};
	\draw [draw,fill=white] 
			(0.25,3) circle (0.75ex);
	\draw[thick] plot [smooth] coordinates {
			(0.75,4) 
			(0.75,3.75) 
			(0.75,3) 
			(0.75,2.25) 
			(1.25,1.75) 
			(1.25, 1.25) 
			(1.75, 0.75) 
			(1.75,0.25) 
			(1.75,0)
	};
	\draw [draw,fill=white] 
			(0.75,3) circle (0.75ex);
	\draw [draw,fill=white] 
			(1.25, 1.25) circle (0.75ex);
	\draw[thick] plot [smooth] coordinates {
			(1.75,4) 
			(1.75,1.75) 
			(1.75,1.25) 
			(2.25,0.75) 
			(2.25,0.25) 
			(2.25,0)
	};
	\draw [draw,fill=white] 
			(2.25,0.75) circle (0.75ex);
	\draw[draw=white,double=black,very thick] plot [smooth] coordinates { 
			(3.25,4) 
			(3.25,3.75) 
			(3.25,3.25) 
			(2.75,2.75) 
			(2.75,2.25) 
			(2.25,1.75) 
			(2.25,1.25) 
			(1.25,0.75) 
			(1.25,0.25) 
			(1.25, 0)
	};
	\draw [draw,fill=white] 
			(2.25,1.25) circle (0.75ex);
	\draw [draw,fill=white] 
			(2.75,2.75) circle (0.75ex);

	\draw[draw=white,double=black,very thick] plot [smooth] coordinates { 
			(2.75,4) 
			(2.75,3.75) 
			(2.75,3.25) 
			(3.25,2.75) 
			(3.25,2.25) 
			(3.25,0)
	};
	\draw [draw,fill=white] 
			(3.25,2.75) circle (0.75ex);

	\draw[dashed,color=green] (0,2.5) rectangle (1,3.5);
	\draw[dashed,color=green] (2.5,2.5) rectangle (3.5,3.5);
	\draw[dashed,color=green] (1,0.25) rectangle (2.5,1.75);
	\draw[dashed] (0,4) -- (3.5,4);
	\draw[dashed] (0,0) -- (3.5,0);
\end{tikzpicture}
\end{lrbox}

\newsavebox\toblowup
\begin{lrbox}{\toblowup}
\begin{tikzpicture} 
	\draw[thick] plot [smooth] coordinates {
			(0.25,4)
			(0.25,3.75)
			(0.5,3)
	};\draw[thick] plot [smooth] coordinates {
			(0.5,3)
			(0.25,2.25)
			(0.25,0)
	};\draw[thick] plot [smooth] coordinates {
			(0.75,4) 
			(0.75,3.75) 
			(0.5,3) 
	};\draw[thick] plot [smooth] coordinates {
			(0.5,3)
			(1.75,1) 
	};\draw[thick] plot [smooth] coordinates {
			(1.75, 1) 
			(1.75,0.25) 
			(1.75,0)
	};\draw[thick] plot [smooth] coordinates {
			(1.75,4) 
			(1.75,0)
	};\draw[thick] plot [smooth] coordinates {
			(1.75,1) 
			(2.25,0.25) 
			(2.25,0)
	};\draw[thick] plot [smooth] coordinates {
			(2.75,4)
			(2.75,3.75)
			(3,3)
	};\draw[thick] plot [smooth] coordinates {
			(3.25,4)
			(3.25,3.75)
			(3,3)
	};\draw[thick] plot [smooth] coordinates { 
			(3,3)
			(1.75,1)
	};\draw[thick] plot [smooth] coordinates {
			(1.75,1) 
			(1.25,0.25) 
			(1.25,0)
	};\draw[thick] plot [smooth] coordinates { 
			(3,3)
			(3.25,2.25) 
			(3.25,0)
	};

	\draw [draw,fill=white] 
			(0.5,3) circle (0.75ex);
	\draw [draw,fill=white] 
			(3,3) circle (0.75ex);
	\draw [draw,fill=white] 
			(1.75, 1) circle (0.75ex);
	\draw[dashed] (0,4) -- (3.5,4);
	\draw[dashed] (0,0) -- (3.5,0);
\end{tikzpicture}
\end{lrbox}

\newsavebox\erasedboundingbox
\begin{lrbox}{\erasedboundingbox}
\begin{tikzpicture} 
	\draw[thick] plot [smooth] coordinates {
			(0.25,3.5) 
			(0.25,1.5)
	};
	\draw [draw,fill=white] 
			(0.25,2) circle (0.75ex);
	\draw[thick] plot [smooth] coordinates {
			(0.75,3.5)
			(0.75,3)
			(0.75,2) 
			(1.75,1.5)
	};
	\draw [draw,fill=white] 
			(0.75,3) circle (0.75ex);
	\draw [draw,fill=white] 
			(0.75, 2) circle (0.75ex);
	\draw[thick] plot [smooth] coordinates {
			(1.75,3.5) 
			(1.75,2)
			(2.25,1.5)
	};
	\draw [draw,fill=white] 
			(1.75,2) circle (0.75ex);
	\draw[draw=white,double=black,very thick] plot [smooth] coordinates { 
			(3.25,3.5)
			(2.75,3)
			(2.75,2)
			(1.25, 1.5)
	};
	\draw [draw,fill=white] 
			(2.75,3) circle (0.75ex);
	\draw [draw,fill=white] 
			(2.75,2) circle (0.75ex);

	\draw[draw=white,double=black,very thick] plot [smooth] coordinates { 
			(2.75,3.5)
			(3.25,3)
			(3.25,1.5)
	};
	\draw [draw,fill=white] 
			(3.25,3) circle (0.75ex);

	\draw[dashed] (0,3.5) -- (3.5,3.5);
	\draw[dashed] (0,1.5) -- (3.5,1.5);
\end{tikzpicture}
\end{lrbox}

\newsavebox\collapsedvertices
\begin{lrbox}{\collapsedvertices}
\begin{tikzpicture} 
	\draw[thick] plot [smooth] coordinates {
			(0.25,3.5) 
			(0.25,1.5)
	};
	\draw [draw,fill=white] 
			(0.25,2) circle (0.75ex);
	\draw[thick] plot [smooth] coordinates {
			(0.75,3.5)
			(0.75,3)
			(0.75,2.5)
			(0.75,2) 
			(1.75,1.5)
	};
	\draw [draw,fill=white] 
			(0.75,2.5) circle (0.75ex);
	\draw[thick] plot [smooth] coordinates {
			(1.75,3.5) 
			(1.75,2)
			(2.25,1.5)
	};
	\draw [draw,fill=white] 
			(1.75,2) circle (0.75ex);
	\draw[draw=white,double=black,very thick] plot [smooth] coordinates { 
			(3.25,3.5)
			(2.75,3)
			(2.75,2.5)
			(2.75,2)
			(1.25, 1.5)
	};
	\draw [draw,fill=white] 
			(2.75,2.5) circle (0.75ex);

	\draw[draw=white,double=black,very thick] plot [smooth] coordinates { 
			(2.75,3.5)
			(3.25,3)
			(3.25,1.5)
	};
	\draw [draw,fill=white] 
			(3.25,3) circle (0.75ex);

	\draw[dashed] (0,3.5) -- (3.5,3.5);
	\draw[dashed] (0,1.5) -- (3.5,1.5);
\end{tikzpicture}
\end{lrbox}

\newsavebox\corollafives
\begin{lrbox}{\corollafives}
\begin{tikzpicture} 
	\draw[thick] plot coordinates {
			(0.25,3.5) 
			(1.75,2.5)
			(0.25,1.5)
	};\draw[thick] plot coordinates {
			(0.75,3.5)
			(1.75,2.5)
			(1.75,1.5)
	};\draw[thick] plot coordinates {
			(1.75,3.5) 
			(1.75,2.5)
			(2.25,1.5)
	};\draw[thick] plot coordinates { 
			(3.25,3.5)
			(1.75,2.5)
			(1.25, 1.5)
	};\draw[thick] plot coordinates { 
			(2.75,3.5)
			(1.75,2.5)
			(3.25,1.5)
	};
	\draw [draw,fill=white] 
			(1.75,2.5) circle (0.75ex);
	\draw[dashed] (0,3.5) -- (3.5,3.5);
	\draw[dashed] (0,1.5) -- (3.5,1.5);
\end{tikzpicture}
\end{lrbox}

%% file: intro.tex

\section{Introduction} 

The aim of this paper is to provide a simplicial model for ``higher props.'' 
A prop, first introduced by Adams and MacLane \cite{catalg}, is a device which is capable of modeling algebraic, coalgebraic, and bialgebraic structures.
Examples of structures which are governed by props include the usual types of (co)algebras (associative, commutative, Lie, Poisson, etc.) and bialgebras (Hopf and Frobenius algebras), but also topological examples such as $n$-fold loop spaces \cite{bv} and topological or conformal field theories \cite{segal88}.
The colored, or multi-sorted, version of prop will be considered here, which allows one to model (among other things) morphisms and families of such structures.
A (colored) prop, like a category, has a set of objects $\{x, y, \dots\}$, but we replace arrows $x\rightarrow y$ by multilinear operations $x_{1}\otimes \dots \otimes x_{n}\rightarrow y_{1}\otimes \dots \otimes y_{m}$.
Operads \cite{geometry,mss}, categories \cite{maclane}, colored operads \cite{bv,bmresolution}, and properads \cite{vallette} are all essentially examples of such props.
See the survey paper \cite{operadsandprops} for more examples and motivation.

The adjective ``higher'' refers to higher category theory, in the $(\infty, 1)$-categorical sense \cite{juliesurvey}.
Roughly, an $(\infty, 1)$-category is to a category as an $A_\infty$-space is to a monoid.
There are many ways to formalize this concept, including simplicial categories \cite{bergner}, complete Segal spaces \cite{rezk}, and quasi-categories \cite{joyal1, htt}.
These are all related by chains of Quillen equivalences of model categories (see the survey article \cite{juliesurvey} for an overview).
Similar definitions of $\infty$-operads were developed by Moerdijk, Cisniski and Weiss in \cite{mw,mw2, cm-ho,cm-ds,cm-simpop}.

This paper is part of a multistage project developing the notion of higher props by extending one of these known models for $\infty$-categories to props. More explicitly, we develop the simplicial model for higher props by showing that the category of simplicial props admits a cofibrantly generated model category structure. The main technical effort of this paper is contained in Proposition~\ref{The Hard Part} where we analyze pushouts of simplicial props with varying sets of colors.  While pushouts of operads can be computed in the category of symmetric sequences, this is not the case for pushouts of props.  This complication requires us to provide significantly different proofs from those in the operad setting \cite{cm-simpop, robertson1}.
Further, pushouts of simplicial props cannot be computed using the later technology of Batanin and Berger \cite{bb}.

The model structure developed in this paper is used in the follow-up paper \cite{hry-2} to give a model structure on the category of simplicial \emph{properads}.
This is important in a parallel project with Donald Yau, initiated in \cite{hry-book}, to understand the homotopy theory of infinity properads.
Later papers in this series will develop combinatorial models for up-to-homotopy props, following the dendroidal approach to higher operads \cite{cm-ds, cm-simpop, cm-ho, mw, mw2} as well as Francis and Lurie's approach to $\infty$-operads in \cite{francis, higheralgebra}.

\subsection{Organization}
Sections \ref{section graphs and megagraphs} -- \ref{section:model structures} are background material, section \ref{S:modelprop} contains the main theorem and its proof, modulo the proof of Lemma \ref{NAFA} which comprises the last four sections.

The paper begins in the next section with the basics of (colored) graphs, megagraphs (parametrized simplicial sets with $\Sigma$-biactions) and graph substitution.
In section \ref{section props} we give a definition of prop, and provide two ways to produce props in \ref{sec_free_prop}, \ref{operadtoprop}.
Section \ref{section:model structures} recalls two specific previously known model structures: the Bergner model structure on simplicially-enriched categories with varying object sets, and the Johnson-Yau model structure on simplicially-enriched props with a fixed set of colors.

In section \ref{S:modelprop} we blend these two model structures together to give a model structure on the category of simplicially-enriched props with varying color sets.
We describe sets of generating (acyclic) cofibrations (definitions \ref{generatingcofibrations}, \ref{generatingacycliccofibrations}) and in \ref{subsection classification fibrations} compute the orthogonal complements of their saturations (the acyclic fibrations and fibrations, respectively).

At this point we go about showing that pushouts of generating acyclic cofibrations are weak equivalences.
Those which come from the generating acyclic cofibrations in $\sSet$ are relatively straightforward, while the remaining ones constitute a major difficulty. In particular, the proof of Lemma \ref{NAFA} takes up the bulk of this paper. 
It begins in section \ref{S:model for pushouts}, where we set up a general categorical framework for modeling pushouts of simplicial props. 
In section \ref{section specialization}, we restrict attention to the types of pushout from Lemma \ref{NAFA}, and are able to make several reductions in our categorical framework.
We use this in section \ref{section local filtration} to give a filtration of this pushout, while in section \ref{section filtration layers} we exhibit each filtration inclusion as a pushout of an acyclic cofibration of $\sSet$. 
This section concludes with a proof of Lemma \ref{NAFA}.

With Lemma \ref{NAFA} proved, we show in the remainder of \ref{subsection relative cell complexes} that relative $J$-cell complexes are weak equivalences, and then use this fact as input in the recognition theorem for cofibrantly generated model categories to establish the desired model structure.
We conclude by noting, in \ref{subsection right proper}, that the model structure on simplicial props is right proper.

\subsection{Acknowledgments}
The first version of this paper lacked an adequate explanation for Lemma~\ref{NAFA}; the authors would like to thank an anonymous referee for pointing this out and would also like to thank D. Yau and M. Batanin for discussions related to this problem.  
This paper has had several referees, and we are thankful to all of them for their insightful and useful comments.
We also thank F. Wierstra for a useful conversation which helped finish section \ref{section filtration layers}.

%% file: props.tex

\section{Graphs and megagraphs}\label{section graphs and megagraphs}
In this paper, we employ the notions of graph and megagraph from \cite{hackneyrobertson1}. We will use descriptive terminology for graphs here; precise formulations may be found in \cite{bb,kock,yj}.
For our purposes, a \emph{graph} is a collection of vertices and directed edges; the tail of each edge is either an \emph{input} of the graph or is connected to a vertex, and similarly the head of each edge is either an \emph{output} of the graph or connected to a vertex. In addition, our graphs have no directed cycles. 
Special examples are the graph $|$ with no vertices and one edge, the graph $\bullet$ with one vertex and no edges, and the empty graph.

\input{megagraph.tex}

\input{graph_sub.tex}

\section{Props}\label{section props}
In this section we define props by using graph substitution, give a description of the free prop on a megagraph, and discuss how every colored operad defines a prop. A good reference for this material is \cite{yj}.

Fix a symmetric monoidal category $(\mathcal{E},\otimes,I)$ with all small limits and colimits, such that $\otimes$ distributes over colimits in both variables, and let $\fC$ be a set of colors.

\begin{notation}
	Suppose $\X$ is a megagraph in $(\mathcal{E},\otimes,I)$ and $G$ is an $X_0$-colored graph.
	Let
	\[
		\X[G] = \bigotimes_{v\in \vertex(G)} X\profilev
	\]
	be the space of decorations of the vertices of $G$ by elements of $\X$.
	Here we are using the \emph{unordered tensor product}; namely, if $S$ is a finite set with $k$ elements and $\{ Z_s \}_{s\in S}$ is an $S$-indexed family of objects of $\mathcal E$, then the unordered tensor product is defined to be
	\[
		\bigotimes_{s\in S} Z_s := \left( \coprod_{\alpha : \{1, \dots, k \} \overset\cong\to S}  Z_{\alpha(1)} \otimes \dots \otimes Z_{\alpha(k)} \right)_{\Sigma_k}
	\]
	where $\alpha$ ranges over all bijections between $\{ 1, \dots, k\}$ and $S$.
	Note that for any fixed bijection $\alpha: \{1, \dots, k \} \to S$, the map $\bigotimes_{i=1}^k Z_{\alpha(i)} \to \bigotimes_{s\in S} Z_s$ is an isomorphism.
\end{notation}

\begin{definition}
\label{ccoloredgprop}
A \emph{$\fC$-colored prop in $\mathcal{E}$} consists of 
\begin{enumerate}
\item
an $\mathcal E$ megagraph $\T$ over $\fC$ (definition \ref{E megagraph}), and
\item
for each graph $G \in \sgc$, a map $\gamma_G$
\[
\begin{tikzcd}
\T[G] = \bigotimes_{v\in G} \T\profilev \dar{\gamma_G} \\ \T\profileg \in \mathcal{E}
\end{tikzcd}\]
such that the following two conditions hold.
\begin{description}
\item[Unity]
$\gamma_C = \id$ whenever $C$ is a standard corolla \eqref{standard_corolla}.
\item[Associativity]
$\gamma$ is associative with respect to graph substitution, in the sense that the diagram
\begin{equation}
\label{gpropass}
\begin{tikzcd}
\bigotimes_{v \in G} \T[H_v] \rar{\otimes \gamma_{H_v}} 
& \T[G] \dar{\gamma_G}\\
\T[G\{H_v\}] \uar{\cong} \rar{\gamma_{G(H_v)}} & \T\profileg
\end{tikzcd}
\end{equation}
is commutative whenever it is defined.
\end{description}
\end{enumerate}
\end{definition}

\begin{notation} When we discuss a prop $\T$ we will refer to $\fC$ as the set of colors of $\T$. When we are dealing with several props it is convenient to denote $\fC$ as $\col(\T)$ so it does not become confusing which set of colors belongs to which prop. 
\end{notation} 

\begin{definition} A morphism of props $f:\R\rightarrow\T$ 
is a morphism of megagraphs so that
\[ \begin{tikzcd}
\R[G] \rar \dar & \T[fG] \dar \\
\R\dch \rar & \T \fdc
\end{tikzcd} \]
commutes for every $\col(\R)$-graph $G$ (where $fG$ is defined as in defintion \ref{definition of fG}).
The category of props and prop morphisms is denoted $\Prop$. 
\end{definition}


In this paper we will be primarily interested when the underlying symmetric monoidal category $\mathcal{E}$ is the category of simplicial sets with monoidal product the Cartesian product. We will write $\sProp$ for the category of props enriched in simplicial sets. The category $\sProp$ inherits both colimits and limits from taking the corresponding (co)limits levelwise in $\Prop$; further, limits in $\sProp$ are created componentwise in $\sSet$.

\begin{remark}
The definition of (colored) prop we use is weaker than classical definitions \`a la Adams-MacLane.
In the classical definition of prop, the set of operations in bi-arity $(0,0)$ form an abelian group by an Eckmann-Hilton argument.
That is, in the classical setting, there are two separate operations
\begin{equation}\label{varnothing eq}
	\T(\varnothing;\varnothing) \times \T(\varnothing;\varnothing) \to \T(\varnothing;\varnothing)
\end{equation}
given by horizontal and vertical composition; in the graph setting there is but a single graph with two vertices and no edges, so our formalism using graph substitution gives only single morphism \eqref{varnothing eq}, which need not be commutative.
We are grateful to M. Batanin for pointing this out to us; more details may be found in \cite[\S 10]{bb}.
This distinction is irrelevant when considering props which come from operads or properads, or when considering props which only have operations with non-empty output\footnote{
i.e.\ props with $\T(\uc; \varnothing) = \varnothing$ for all $\uc \neq \varnothing$ and $\T(\varnothing; \varnothing) = *$.
} 
(or input), such as props modeling nonunital field theories.
\end{remark}

\input{free_prop_s}

\input{colored_ops}

%% file: megagraph.tex
\subsection{Megagraphs}
Consider the free monoid monad $\M: \Set \to \Set$ which takes a set $\fC$ to $\M(\fC) = \coprod_{k \geq 0} \fC^{\times k}$. There are right and left actions of the symmetric groups on the components of $\M \fC$. More compactly we could say that there are both right and left actions of the symmetric \emph{groupoid} $\Sigma =\coprod_{n\geq 0} \Sigma_n$ on $\M \fC$. A $\Sigma$-bimodule is a set with compatible left and right $\Sigma$-actions. 
We now describe an extension of the notion of graph, namely one in which edges are permitted to have multiple inputs and outputs. %

\begin{definition}

\begin{enumerate} 

\item A \emph{megagraph} $\X$ consists of a set of `objects' $X_0$,
a set of `arrows' $X_1$, two functions $s: X_1 \to \M X_0$ and $t: X_1 \to
\M X_0$, and  right and left $\Sigma$ actions on $X_1$.
These actions should have an interchange property $\tau \cdot (x\cdot
\sigma) = (\tau \cdot x) \cdot \sigma$ and should be compatible with
those on $\M X_0$, so $t( \tau \cdot x) = \tau \cdot t(x)$ and  $s(  x
\cdot \sigma) = s(x) \cdot \sigma$, where $\sigma, \tau \in \Sigma$.
\item 
A map of megagraphs $f: \X \to \Y$ is determined by maps $f_0: X_0 \to
Y_0$ and $f_1: X_1 \to Y_1$ so $sf_1 = (\M f_0)s$ and $tf_1 = (\M f_0)t$.
The collection of megagraphs determines a category which we
call $\Mega$.

\item 
A \emph{simplicial megagraph} will be a structure with a \emph{discrete} simplicial set of objects $X_0$, a simplicial set of arrows $X_1$, along with a compatible structure of a megagraph on each $(X_1)_n$.

\item 
For a map of simplicial megagraphs we require that the function $f_1$ be a map of simplicial sets; we denote by $\sMega$ the category of simplicial megagraphs.
\end{enumerate} 
 \end{definition}

\begin{remark}
	We will often denote the set of objects $X_0$ of a megagraph $\X$ by $X_0 = \fC$. In this case, we will also call $\X$ a megagraph over $\fC$.
\end{remark}

\begin{definition}
	In anticipation of the role played in the definition of prop (\ref{ccoloredgprop}, we will call elements of 
	$\M(\fC) \times \M(\fC)$ \emph{input-output profiles of $\fC$}. We will use underlines to give the shorthand notation
	\[
		\dch = (c_1, \dots, c_m; d_1, \dots, d_n) \in \fC^{\times m} \times \fC^{\times n} \subset \M(\fC) \times \M(\fC)
	\]
	whenever we write down an input-output profile.
\end{definition}

\begin{remark}
	If $\X$ is a (simplicial) megagraph and $\dch$ is an input-output profile of $X_0$, we denote by $\X\dch$ the preimage of $\dch$ under $(s,t)$; in other words, the diagram 
	\[ \begin{tikzcd} 
		\X\dch \rar \dar[hook] & \{ \dch \} \dar[hook]
	\\
		X_1 \rar{(s,t)} & \M(X_0) \times \M(X_0)
	\end{tikzcd} \]
	is a pullback.
	The induced map
	\[
		\coprod_{\dch \in \M(X_0) \times \M(X_0)} \X\dch \to X_1
	\]
	is a bijection, hence an isomorphism. 
\end{remark}

This inspires the following definition of megagraph, which is valid in an arbitrary category $\mathcal E$.

\begin{definition}\label{E megagraph}
	Let $\fC$ be a set. An $\mathcal E$-megagraph $\X$ over $\fC$ consists of the data
	\begin{itemize}
		\item an object $\X\dch$ of $\mathcal E$ for each input-output profile $\dch$ of $\fC$,
		\item $\mathcal E$-maps $\tau_* : \X\dch \to \X(\tau \cdot \uc; \ud)$ for each $\tau \in \Sigma_{|\uc|}$, and
		\item $\mathcal E$-maps $\sigma^* : \X \dch \to \X(\uc; \ud \cdot \sigma)$ for each $\sigma \in \Sigma_{|\ud|}$.
	\end{itemize}
	These data should satisfy $\tau_* \sigma^* = \sigma^* \tau_*$, $\sigma^* (\sigma')^* = (\sigma' \sigma)^*$, and $\tau_* \tau'_* = (\tau \tau')_*$.
	If $\X$ is an $\mathcal E$-megagraph over $\fC$ and  $\Y$ is an $\mathcal E$-megagraph over $\fD$, then a morphism $\X \to \Y$ consists of a set map $f: \fC \to \fD$ and $\mathcal E$-morphisms
	\[
		\X\dch \to \Y\fdc
	\]
	which are compatible with the $\Sigma$ actions.
\end{definition}

Thus we may think of a megagraph as a family of objects parametrized by $\M (\fC) \times \M (\fC)$, together with left and right $\Sigma$-actions. 
A megagraph $\X$ over $\fC$ would be called a $\fC$-colored $\Sigma$-bimodule in \cite{fmy}.

%% file: graph_sub.tex
\subsection{Graph Substitution}\label{graph sub}

Graph substitution is a powerful notational tool which we employ throught this paper.

\begin{definition}
Let $\fC$ be a set. A \emph{$\fC$-colored graph} is a graph together with the following additional structure.
\begin{itemize}
	\item An edge coloring function $\colorFunction: \edge(G)\to \fC$,
	\item for each vertex $v\in \vertex(G)$ a fixed ordering on the input edges $\inp(v)$ and the output edges $\out(v)$, and
	\item an ordering on both $\inp(G)$ and $\out(G)$.
\end{itemize}
Let $\sgc\dc$ be the set of $\fC$ colored graphs so that $\colorFunction(\inp(G)) = \underline{c}$ and $\colorFunction(\out(G)) = \underline{d}$ (as ordered sets).
\end{definition} 

This is a subset of the data comprising a `decoration of $G$' in \cite[Def. 41]{hackneyrobertson1}. 

\begin{definition}\label{definition of fG}
	If $G$ is a $\fC$-colored graph and $f: \fC \to \fD$ is a map of sets, then denote by $fG$ the $\fD$-colored graph with all of the same structure as $G$, except that the edge coloring function is changed to the composite
	\[
		\edge(G) \overset\colorFunction\to \fC \overset{f}\to \fD.
	\]
\end{definition}

For each $c\in \fC$, write $|_c$ for the graph with one $c$-colored edge and no vertices.
Given an input-output profile $\dch = c_1, \dots, c_n; d_1, \dots, d_m$, the \emph{standard corolla} $C_{\dch}$ is the graph with edges $e_1,\dots,e_{n+m}$, a single vertex $v$ adjacent to all edges, and, as ordered lists,
\begin{equation}\label{standard_corolla}
\begin{aligned}
	\inp(G) = \inp(v) &= e_1, \dots, e_n & \colorFunction(\inp(G)) &= c_1,\dots, c_n \\
	\out(G) = \out(v) &= e_{n+1}, \dots, e_{n+m} & \colorFunction(\out(G)) &= d_1, \dots, d_m.
\end{aligned}
\end{equation}

Suppose now that we have a graph $G \in \sgc\dc$, a vertex $v\in G$, and a graph $K \in \sgc ( \colorFunction(\inp(v)) ; \colorFunction(\out(v)) )$. 

\begin{definition} 
The \emph{graph substitution} $G(K)$ is the graph with
\begin{align*}
	\edge(G(K)) &= (\edge(G) \amalg \edge(K)) / (\out(v),\inp(v) \sim \out(K),\inp(K)) \\
	\vertex(G(K)) &= (\vertex(G) \setminus \{v\}) \amalg \vertex(K)
\end{align*}
together with all the data of the evident orderings, coloring functions, source, target, etc. 
\end{definition} 

\begin{remark}  We can, in fact, do multiple graph substitutions at once, which we will write as $G\{K_v\}$ where $v$ ranges over some subset of vertices of $G$. An example is given in figure~\ref{graph sub ex}, where the orderings should all be taken as those coming from the plane.
\end{remark} 

\begin{figure}
	\includegraphics[scale=0.4]{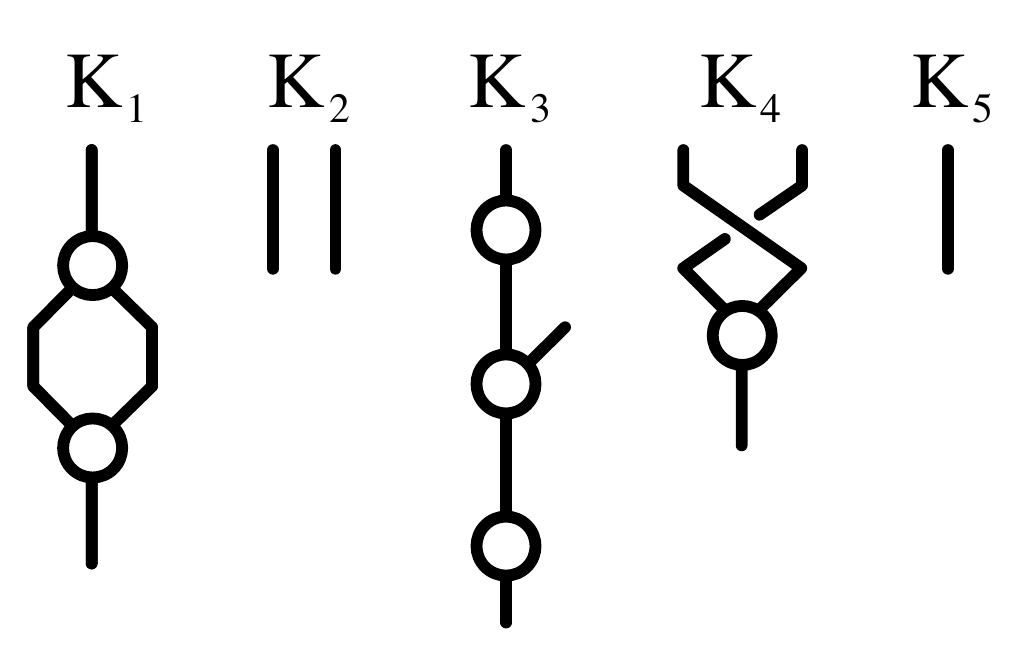}
	\includegraphics[scale=0.4]{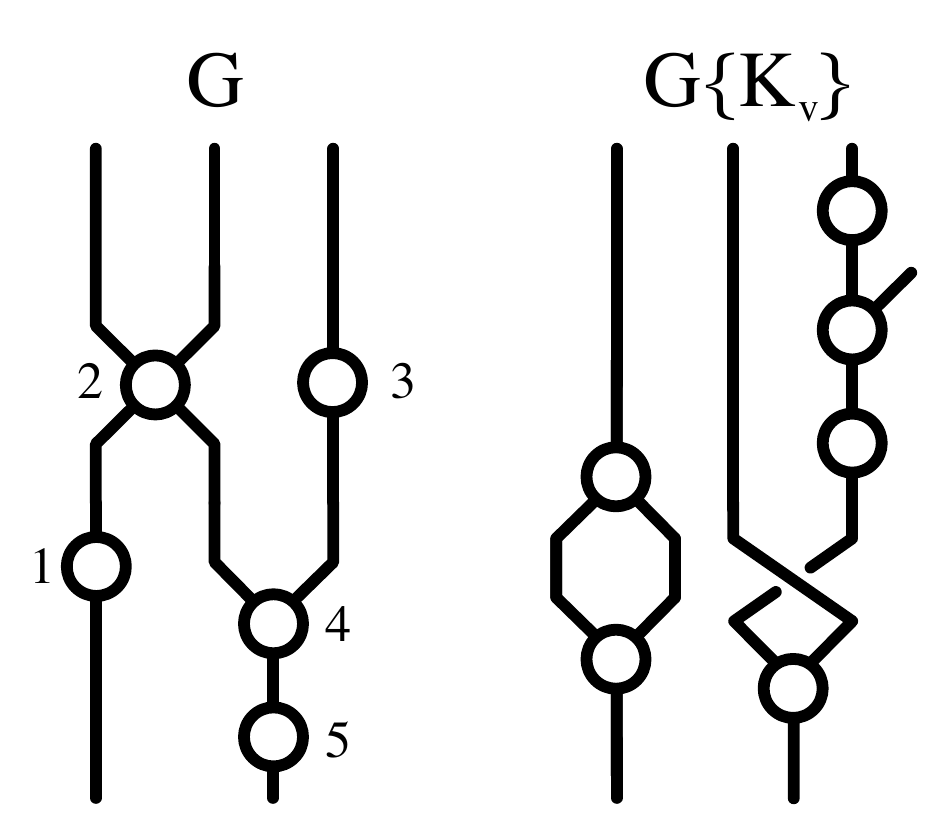}
	\caption{A sample graph substitution}\label{graph sub ex}
\end{figure}

\begin{notation}
We avoid explicitly writing $\colorFunction$ unless absolutely necessary. In particular, we will write $\sgc \profilev$ instead of $\sgc ( \colorFunction(\inp(v)) ; \colorFunction(\out(v)) )$.
\end{notation}

Graph substitution is associative, so $\sgc\dc$ is in fact a category, with an arrow 
\[
	G\{K_v\} \to G
\]
for every graph substitution.

\begin{example}\label{ex:non_triv_auto}
	Let $\fC = *$, let $G \in \sgc(*;*)$ be the graph with two vertices $v,w$ and four edges depicted in figure~\ref{fig:non_triv_auto}. 
	Let $K_w$ and $K_v$ be the relevant non-standard corollas. Then
	up to renaming of edges, $G\{ K_v, K_w \} = G$.
	This is a non-trivial graph substitution, hence is an example of a non-trivial automorphism in $\sgc(*;*)$.
	
	\begin{figure}
	\includegraphics[scale=0.4]{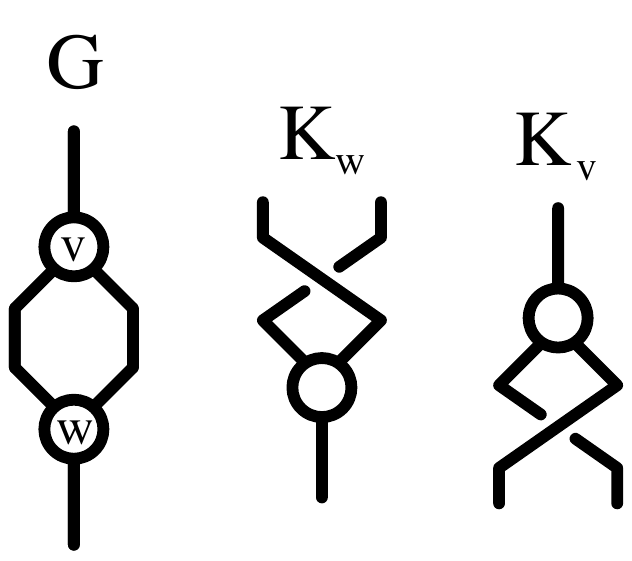}
	\caption{Graphs from example \ref{ex:non_triv_auto}}\label{fig:non_triv_auto}
\end{figure}
\end{example}

\begin{proposition}\label{terminality of corolla}
	The standard corolla $C_{\dch}$ is a terminal object in $\sgc\dc$.
\end{proposition}
\begin{proof}
	First notice that if $K \in \sgc\dc$, then $K=C_{\dch}(K)$, where $K$ is substituted into the unique vertex of the corolla, whence there is map $K\to C_{\dch}$.
	We now just need to check that there are no non-identity automorphisms of $C_{\dch}$:
	all maps to $C_{\dch}$ are of the form
	\[
		C_{\dch}(K_v) \overset{f}{\to} C_{\dch},
	\]
	where $v$ is the unique vertex of $v$ and $K_v \in \sgc\dc$ is some graph. 
	But $C_{\dch}(K_v) = K_v$, hence if $f$ is an automorphism then 
	$f$ is the identity on $C_{\dch}(C_{\dch}) = C_{\dch}$.
\end{proof}

%% file: free_prop_s.tex
\subsection{The free prop associated to a megagraph}\label{sec_free_prop}

\begin{definition} The forgetful functor $U:\Prop\rightarrow \Mega$ applied to a prop $\T$ is defined by \[
(U\T)_0 = \col \T \qquad \qquad (U\T)_1 = \coprod_{\underline{a},
\underline{b} \in \M(\col \T)} \T(\underline a ; \underline b) \] with
the induced source and target maps. 
\end{definition} 

\begin{thm}[{\cite[Theorem 14]{hackneyrobertson1}}]
\label{T:freeprop}The functor $U: \Prop \to \Mega$ has a left adjoint $F: \Mega \to \Prop$. \end{thm}

For a concrete  description of $F\X$, see Proposition~\ref{free prop description}. We will repeatedly use the following class of free props in section~\ref{S:modelprop}. 

\begin{definition} The $(n,m)$-corolla $\mathcal{G}_{n,m}$ is the connected graph with one vertex, $n$ inputs and $m$ outputs, and this graph can be naturally modeled by a megagraph $\X$ by letting
\begin{align*}
        X_0 &= \set{a_1, \dots, a_n, b_1, \dots, b_m} &
        s(e,e) &= (a_1, \dots, a_n) \\
        X_1 &= \Sigma_n \times \Sigma_m &
        t(e,e) &= (b_1, \dots, b_m)
\end{align*}
together with the evident symmetric group actions on $X_1$.
\end{definition} 

\begin{definition}\label{D:gnm} If $\X$ is the megagraph described above, which models the $(n,m)$-corolla $\mathcal{G}_{n,m}$, we will also write $\mathcal{G}_{n,m}$ for the prop $F(\X)$. \end{definition}

The prop $\mathcal{G}_{n,m}$ has a single generating operation in $\mathcal{G}_{n,m}(a_1, \dots, a_n; b_1, \dots, b_m)$.

\begin{theorem} \label{sgc is prop}
The collection \[
	\left\{\sgc\dc\right\}
\] forms a $\fC$-colored prop enriched in $\Cat$, with prop structure obtained from graph substitution.
\end{theorem}
\begin{proof}
	The relevant properties are established in \cite[\S 6.3]{yj}.
\end{proof}

The free-forgetful adjunction
\[
	F : \Mega \rightleftarrows \Prop : U
\]
from Theorem \ref{T:freeprop} has an associated comonad $\bot = FU: \Prop \to \Prop$ and an associated monad $\top = UF : \Mega \to \Mega$. In this section we will describe $\Prop$ as the category of algebras for the monad $\top$.

If $\X$ is the underlying megagraph of a prop $\T$ and $G \in \sgc\dc$, we have the map
\[
	\gamma_G: \T[G] \to (\top \T)\dc \to \T \dc = \T ( \inp(G) ; \out(G) )
\]
since $\T$ is an algebra for the monad $\top$. 


\begin{proposition}\label{colim sgc}
 If $\T$ is a simplicial $\fC$-prop, then there exists a functor, called the a \emph{decoration functor}
\[
	\partial: \sgc\dc \to \sSet
\]
which sends a $\fC$-colored graph $G$ to the simplicial set $\T[G]$.
Furthermore, $\partial$ has the property that $\colim \partial \cong \T\dc$.
\end{proposition}

\begin{proof}
For each morphism $f:G\{K_v\} \to G$ in $\sgc\dc$, we declare $\partial(f)$ to be
\begin{align*}
	\T[G\{K_v\}] &= \prod_{w\in G\{K_v\}} \T ( \inp(w) ; \out(w)) \\
	&= \prod_{v\in G} \prod_{w\in K_v} \T ( \inp(w) ; \out(w)) \\
	&\overset{\prod \gamma_{K_v}}{\to} \prod_{v\in G} \T( \inp(K_v) ; \out(K_v)) = \prod_{v\in G} \T \profilev = \T[G].
\end{align*}

To see that $\colim \partial \cong \T\dc$, recall from Proposition \ref{terminality of corolla} that the standard corolla \eqref{standard_corolla} is the terminal object in $\sgc\dc$, so 
\[ \colim \partial = \partial (C_{\dch}) = \T[C_{\dch}] = \T\dc.\]

\end{proof}

\begin{proposition}\label{free prop description}
	Let $\X$ be a $\fC$-colored megagraph, $\dch$ be a fixed $\fC$ input-output profile, and
	\[
		I \subset \Ob \sgc\dc
	\]
	be a set of representatives for isomorphism classes of $\sgc\dc$.
	Then in the free prop $F\X$ we have
	\[
		F\X \dc \cong \coprod_{G \in I}  \X [G].
	\]
\end{proposition}
\begin{proof}
	This is an unraveling of the description of the free prop given in \cite[Appendix]{hackneyrobertson1}.
\end{proof}

%% file: colored_ops.tex
\subsection{The prop generated by an operad}\label{operadtoprop} Another important family of props are colored operads. 
Recall that a colored operad, $\someoperad$, is a
structure with a color set $\col \someoperad$ and hom sets $\someoperad(c_1, \dots, c_n; c)$ for each list of colors $c_1, \dots, c_n,c$, together with appropriate composition operations.
The precise definition is a bit more involved than that of prop since the operadic composition mixes together horizontal and vertical propic compositions (see, for example, \cite{bmresolution}).
We regard colored operads as a special type of prop, namely those which are completely determined by the components with only one output.
Let $\Operad$ be the category of colored operads, which, following the conventions of ~\cite{moerdijklecture}, we refer to as simply
\emph{operads}. 
There is a forgetful
functor \[ U_0: \Prop \to \Operad \] which takes $\T$ to an operad $U_0(\T)$
with $\col U_0(\T) = \col \T$. 
The morphism sets are defined by \[ U_0(\T)
(a_1, \dots, a_n; b) = \T(a_1, \dots, a_n; b) \]  and operadic composition
\[\begin{tikzcd} U_0(\T)(a_1, \dots, a_n; b) \times \prod_{i=1}^n U_0(\T)
(c_{i,1}, \dots , c_{i,p_i}; a_i) \dar{\gamma} \\
U_0(\T)\left( c_{1,1}, \dots, c_{1,p_1},c_{2,1}, \dots, c_{n,p_n}
; b\right) 
\end{tikzcd}\] is given by  \[
\gamma(g, f_1,\dots,f_n) = g\circ_v (f_1 \circ_h \dots \circ_h f_n).\]
In \cite[Proposition 11]{hackneyrobertson1}, we indicate how an operad generates a prop and prove the following result.

\begin{prop}\label{P:operadpropadjunction} The forgetful functor $U_0: \Prop \to\Operad$ has a left adjoint $F_0: \Operad \to \Prop$ with $\col F(\someoperad) =\col \someoperad$.
If $\someoperad$ is an operad then \[ F_0(\someoperad) (a_1, \dots, a_n; b) \cong \someoperad(a_1, \dots, a_n; b). \]
Consequently, $U_0F_0 \cong \id_{\Operad}$.\end{prop}

\noindent The obvious variant holds if one instead considers operads and props enriched in $\sSet$: there is an adjunction $F_0: \sOperad \rightleftarrows \sProp : U_0$, cf. Corollary~\ref{cor42}.